\documentclass[a4paper,10pt, leqno]{amsart}
\usepackage[all]{xy}
\usepackage{graphicx}
\usepackage{amssymb}

\numberwithin{equation}{section}
\newtheorem{Theorem}{Theorem}[section]
\newtheorem{Thm}[Theorem]{Theorem}

\newtheorem{Corollary}[Theorem]{Corollary}

 \newtheorem{Lemma}[Theorem]{Lemma}
\newtheorem{Proposition}[Theorem]{Proposition}
\newtheorem{Prop}[Theorem]{Proposition}
\theoremstyle{remark}
\newtheorem{Remark}[Theorem]{Remark}

\theoremstyle{definition}
\newtheorem{Definition}[Theorem]{Definition}
\newtheorem{Def}[Theorem]{Definition}
\newtheorem{Example}[Theorem]{Example}
\newtheorem{Exa}[Theorem]{Example}
\newtheorem*{acknowledgements}{Acknowledgements}
\newcommand{\inner}[2]{\left\langle{#1},{#2}\right\rangle}

\newcommand{\imag}{\op{i}}
\newcommand{\R}{{\mathbf R}}

\newcommand{\Z}{{\mathbf Z}}
\newcommand{\C}{{\mathbf C}}
\newcommand{\mc}[1]{{\mathcal #1}}
\newcommand{\mb}[1]{{\mathbf #1}}
\newcommand{\pmt}[1]{{\begin{pmatrix} #1  \end{pmatrix}}}

\renewcommand{\phi}{\varphi}
\renewcommand{\epsilon}{\varepsilon}
\newcommand{\op}[1]{{\operatorname{ #1}}}
\newcommand{\dy}{\displaystyle}

\renewcommand{\phi}{\varphi}
\renewcommand{\epsilon}{\varepsilon}

\renewcommand{\mid}{\, ; \,}
\newcommand{\SL}{\op{SL}}
\renewcommand{\phi}{\varphi}
\newcommand{\pt}[1]{\mathrm{#1}}

\title{
Analytic extensions of constant mean curvature
one geometric catenoids in de Sitter $3$-space
}

\author{S.~Fujimori}
 \address[Shoichi Fujimori]{%
  Department of Mathematics, Hiroshima University,
  Higashihiroshima, Hiroshima 739-8526, Japan
}
\email{fujimori@hiroshima-u.ac.jp}

\author{Y. Kawakami}
 \address[Yu Kawakami]{%
   Faculty of Mathematics and Physics,
   Kanazawa University,
   Kanazawa, 920-1192, Japan,
}
\email{y-kwkami@se.kanazawa-u.ac.jp}

\author{M. Kokubu}
\address[Masatoshi Kokubu]{%
   Department of Mathematics, School of Engineering, 
   Tokyo Denki University, 
   Tokyo 120-8551, Japan}
\email{kokubu@cck.dendai.ac.jp}

\author{W.~Rossman}
\address[Wayne Rossman]{%
   Department of Mathematics, Faculty of Science,
   Kobe University,
   Rokko, Kobe 657-8501, Japan
}
\email{wayne@math.kobe-u.ac.jp}

\author{M.~Umehara}
\address[Masaaki Umehara]{%
   Department of Mathematical and Computing Sciences,
   Tokyo Institute of Technology
   2-12-1-W8-34, O-okayama, Meguro-ku,
   Tokyo 152-8552, Japan.
}
\email{umehara@is.titech.ac.jp}

\author{K.~Yamada}
\address[Kotaro Yamada]{%
   Department of Mathematics\\
   Tokyo Institute of Technology\\
   O-okayama, Meguro, Tokyo 152-8551, 
   Japan
}
\email{kotaro@math.titech.ac.jp}

\author{S.-D. Yang}
\address[Seong-Deog Yang]{%
   Department of Mathematics\\
   Korea University\\ 
   Seoul 136-701, Republic of Korea}
\email{sdyang@korea.ac.kr}

\date{July 06, 2022}
\thanks{The first author was supported in part by Grant-in-Aid for Scientific 
Research (C) No.17K05219 and (C) No. 21K03226
from Japan Society for the Promotion of Science.}
\thanks{The second author was supported in part by Grant-in-Aid for Scientific Research  (C) No.19K03463 from Japan Society for the Promotion of Science. }
\thanks{The third author was supported in part by Grant-in-Aid for Scientific Research (C) No. 17K05227 and (C) No.20K03617 from Japan Society for the Promotion of Science. }
\thanks{The fourth author was supported in part by Grant-in-Aid for Scientific Research (C) No.20K03585 and (S) No. 17H06127 from Japan Society for the Promotion of Science. }
\thanks{The fifth author 
was partially 
supported by Grant-in-Aid for 
Scientific Research (B) No.\ 21H00981, and the fifth and sixth authors were supported in part by Grant-in-Aid for Scientific Research (B) No.17H02839 from Japan Society for the Promotion of Science. }
\thanks{The seventh author was supported in part by NRF2017R1E1A1A03070929 and NRF 2020R1F1A1A01074585 from the National Research Foundation of Korea.}
\usepackage[active]{srcltx}
\begin{document}

\maketitle
\begin{abstract}
We show that a certain simply-stated notion 
of \lq\lq analytic completeness'' of the image of
a real analytic map implies the map admits no analytic extension.
We also give a useful criterion for that notion of analytic completeness by 
defining arc-properness of continuous maps, which can be considered as a 
very weak version of properness.
As an application, we judge the 
analytic completeness of a certain class of constant mean curvature 
surfaces (the so-called ``G-catenoids'') or their analytic 
extensions in the de Sitter 3-space.
\end{abstract}

\section*{Introduction}
In the authors' previous  works
\cite{F}--\cite{FRUYY2}, \cite{UY}--\cite{Y}
on zero mean curvature surfaces 
in the Lorentz-Minkowski $3$-space $\R^3_1$
and constant mean curvature one surfaces
(i.e. CMC-$1$ surfaces) in de 
Sitter 3-space $S^3_1$, a number of 
concrete examples were constructed.
Such surfaces in $\R^3_1$ or
$S^3_1$ have singularities, in general.
Moreover, some of them
have non-trivial analytic extensions
(cf.\ Definition \ref{def:r-A6}).
To find such extensions,
one sometimes needs 
further techniques beyond
the usual reflection principle.

\begin{figure}[hbt]%
 \begin{center}
  \begin{tabular}{c@{\hspace{2em}}c@{\hspace{2em}}c@{\hspace{2em}}c}
       \includegraphics[width=5.0cm]{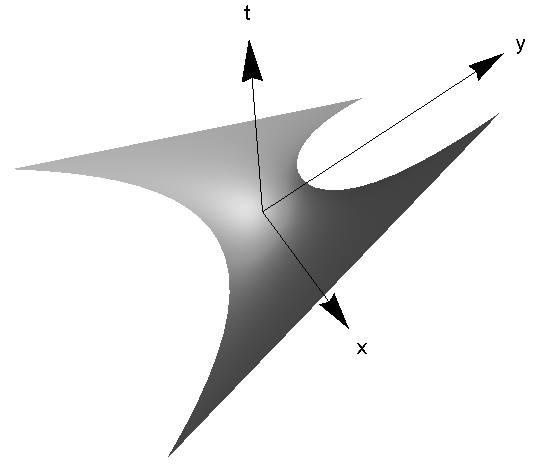} & 
       \includegraphics[width=5.0cm]{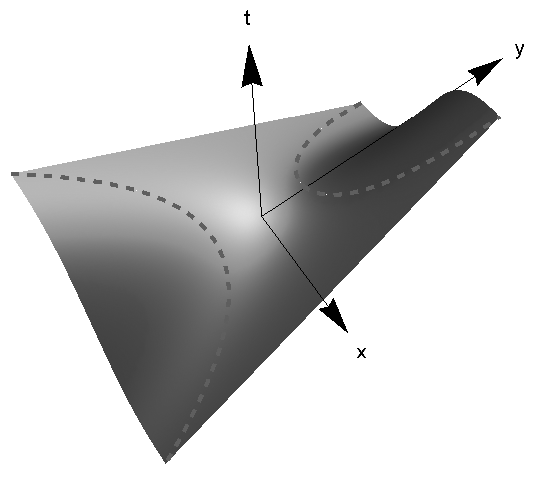} 
  \end{tabular}
 \end{center}
\caption{ The image of $f_K$ (left) and its analytic extension to
a zero mean curvature graph (right)
(the curves where the surface changes type are also indicated).
}%
\label{fig1c}
\end{figure}
Thinking of $\R^3_1$ with signature $(-++)$,
the analytic map
$f_K\colon{}\R\times T^1\to (\R^3_1;t,x,y)$
defined by
\begin{equation}\label{eq:A-cat-max}
f_K(u, v) = (\sinh u \cos v, u, \cosh u \cos v)
\qquad (u\in \R, \,\, v\in T^1),
\end{equation}
gives a typical such example 
(the subscript ``K'' stands for ``Kobayashi'' \cite{K}), 
where
\begin{equation}\label{eq:T1}
T^1:=\R/2\pi \Z
\end{equation}
is a $1$-dimensional torus.
This map $f_K$ gives a parametrization of a maximal surface
derived from the Weierstrass-type representation formula
and has two disjoint 
singular sets along the lines $v=0,\pi$.
The map is a proper  (i.e. the inverse image of a compact set is
compact) analytic map, but
is a two-to-one mapping away from its singular set 
(consisting of
fold singular points).
Osamu Kobayashi \cite{K} found that 
the image of $f_K$ is a proper subset of the 
entire (zero-mean curvature) 
graph of the function
$
t=y \tanh x
$.
The complement of the image of
$f_K$ within this graph contains two disjoint
time-like zero mean curvature surfaces,
that is, the graph of the function $t=y \tanh x$ can be considered as
an analytic extension of the image of $f_K$ 
(cf. Theorem \ref{Cor:EPHK}, 
the precise definition of analytic extension is given in 
Definition \ref{def:r-A6}.)

Up to now several space-like maximal surfaces
admitting analytic extensions have been found
by several geometers, 
including the authors, see
\cite{A, A3, CR, 
FKKRSUYY2, FKKRUY4, HK},
which have motivated the authors to 
find criteria that ensure the resulting 
surfaces satisfy \lq\lq analytic completeness\rq\rq,
namely, they have no further non-trivial 
analytic extensions (cf.  Proposition \ref{prop;Aext}). 

For example, space-like maximal surfaces in $\R^3_1$
which are invariant under a one-parameter family of 
isometries with a common fixed point (i.e. G-catenoids, see below)
are congruent to one of the 
following three surfaces up to a homothety:
\begin{itemize}
 \item the \emph{elliptic G-catenoid}
       \begin{equation}\label{eq:e-catenoid}
f_E(u,v)=(u, \sinh u, u \cos v,\sinh u \sin v) \qquad
        (u,v)\in \R\times T^1,
       \end{equation}
 \item the \emph{parabolic G-catenoid}
       \begin{equation}\label{eq:p-catenoid}
          f_P(u,v):=
          \left(v-\frac{v^3}3+u^2v, v+\frac{v^3}3-u^2v,2uv
                 \right)\qquad (u,v)\in \R^2,
       \end{equation}
\item the \emph{hyperbolic G-catenoid}
       \begin{equation}\label{eq:h-catenoid}
          f_H(u,v)=(\cosh u\sin v, v, \sinh u \sin v)\qquad
(u,v)\in \R^2.
       \end{equation}
\end{itemize}
The image of $f_E$ coincides with the subset
\begin{equation}\label{eq:E0}
\mc E=\{(t,x,y)\in \R^3_1\,;\, x^2+y^2=\sinh^2 t\},
\end{equation}
and is rotationally symmetric with
respect to the time-like axis, and has a cone-like singular point.
Moreover,
$f_E$ has no analytic extensions (in fact,
the map is proper and admits only one cone-like singular point, 
and we can apply Proposition~\ref{prop:first} 
and conclude that $f_E$ is analytically complete). 

As for $f_P$ in \eqref{eq:p-catenoid}, its image 
is not closed.
If we set
\begin{equation}\label{eq:p2}
 \mc P:=\Big\{(t,x,y)\in\R^3_1\,;\,
 12 (x^2 - t^2) = (x + t)^4 - 12 y^2\Big\},
\end{equation}
then it has a cone-like singular point
and satisfies
\begin{equation}\label{eq:P}
\mc P=f_P(\R^2)\cup L,\qquad L:=\Big\{(t,-t,0)\,;\,t\in \R\Big\}.
\end{equation}
The inclusion map associated with $\mc P$ can be considered as an analytic 
extension of the map $f_P$ (cf. Theorem \ref{Cor:EPHK}).
Similarly, $f_H$ also
has an analytic extension.
In fact, by a suitable symmetry $T$ in $\R^3_1$, 
we can write
\begin{equation}\label{eq:H0}
    \mc H:=\Big\{(t,x,y)\in \R^3_1\,;\, \sin^2 x+y^2-t^2=0\Big\}
=\overline{f_H(\R^2)\cup T\circ f_H(\R^2)},
\end{equation}
which is a singly-periodic surface
with a countably infinite number of
cone-like singular points, giving an analytic
extension of $f_H$ (see Examples \ref{ex:H0} and
Theorem \ref{Cor:EPHK}).  

The figures of $\mc P$ and $\mc H$ are found  
in \cite[Fig. 4 (left) and Fig. 1 (left)]{FKKRSUYY2}, 
and they can be expressed as the zero sets of  
the single analytic functions
\begin{align*}
&\tau_P(t,x,y):= 12 (x^2 - t^2) - (x + t)^4 + 12 y^2, \\
&\tau_H(t,x,y):=\sin^2 x+y^2-t^2
\end{align*}
respectively. By this fact, we can show that they are
analytically complete
(cf. Proposition~\ref{prop:first}).
However, in contrast to this characterization of
$\mc P$ and $\mc H$ as the zero sets of
analytic functions, analytically complete subsets
might not be expressed as the zero set of 
a certain analytic function, in general:
For example, the image of the map
\begin{equation}
f_0(t):=(t,e^{-1/t})\qquad (t>0)
\end{equation}
is an analytically complete subset of $\R^2$
(cf. Example \ref{ex:Kok})
although it cannot be characterized as
the zero set of any real analytic function.
In fact, 
$f_0((0,\infty))$ is not a closed subset of $\R^2$. 

In this paper, we define
analytic completeness
for subsets of real analytic manifolds, 
which is essentially the same concept as
the arc-symmetric property 
introduced by Kurdyka \cite{arc} (see also Nash \cite{N}
for arc-structures).
Moreover, inspired by \cite{arc},
we introduce 
\lq\lq arc-properness\rq\rq\ for 
continuous maps to an analytic manifold
(cf.~Definition~\ref{def:ap}).
One may regard arc-properness as a
very weak version of properness of maps (cf. Proposition \ref{prop:P}).  
In fact, proper maps are arc-proper, but the
converse is not true, in general.
This concept plays an important
role in describing our criterion
for analytic completeness (cf. Theorem~\ref{thm:first}).
In the remainder of this paper,
we shall demonstrate that this criterion 
is actually applicable to concrete examples. 
In fact, in the last two sections of this paper, 
it is applied 
to show the analytic completeness 
of analytic extensions of constant
mean curvature one catenoids in de Sitter $3$-space,
as follows:
Let $\R^4_1$ be the Lorentz-Minkowski $4$-space
with the metric $\inner{~}{~}$ of 
signature $(-++\,+)$.
Then  
\[
  S^3_1=
  \Big\{X=(x_0,x_1,x_2,x_3) \in\R^4_1\, ; 
\, (\inner{X}{X}:=)
-(x_0)^2+\sum_{j=1}^3 (x_j)^2=1\Big\},
\]
with the metric induced from $\R^4_1$, 
is a simply-connected Lorentzian 
$3$-manifold with constant sectional curvature $1$,
called the {\it de Sitter 3-space}.

Space-like constant mean curvature one surfaces in $S^3_1$ 
have a Weierstrass-type representation formula
like space-like maximal
surfaces in Lorentz-Minkowski space $\R^3_1$
(see \cite{AA}, \cite{F} and \cite{FRUYY}).
We consider constant mean curvature one surfaces
in $S^3_1$
and zero mean curvature surfaces in $\R^3_1$ with either of the
following properties:

\begin{itemize}
\item[(G)]
The surface is invariant under a one-parameter family of
isometries with a common fixed point.
\item[(W)]
The surface is weakly complete with Gauss map of degree $1$ and Hopf differential having a pole of order $2$ at each of two ends.
\end{itemize}
We call them geometric catenoids (G-catenoids in short) if they satisfy
(G), and catenoids with the same holomorphic data as the
Euclidean catenoid (W-catenoids in short) if they satisfy (W).
Here the subscript ``W'' stands for ``Weierstrass-type representation formulas''
for space-like constant mean curvature one surfaces
in $S^3_1$ and zero mean curvature surfaces in $\R^3_1$.

As mentioned above, G-catenoids in $\R^3_1$
are space-like maximal surfaces 
that are congruent to one of $f_E,\, f_P,\, f_H$ up to a homothety,
and the
last two maps have analytic extensions as mentioned above. 
W-catenoids in $\R^3_1$
are congruent to either $f_E$ or $f_K$ up to homotheties. 
(We have already seen that $f_K$ can be
extended as an entire graph.)

On the other hand, in \cite{Y} and \cite{FKKRUY2}, 
all G-catenoids and W-catenoids in $S^3_1$ 
are classified.
In light of this, it is interesting that
there are so many different types of
G-catenoids and W-catenoids in $S^3_1$.
G-catenoids in $S^3_1$ admit only cone-like singular points
like as in the case of catenoids in $\R^3_1$.
On the other hand, W-catenoids in $S^3_1$ admit 
various kinds of singular points.

In this paper, we focus on G-catenoids,
each of which is induced by
a family of isometries fixing a geodesic of $S^3_1$.
(Analytic completeness of W-catenoids will be
discussed in the forthcoming paper \cite{FKKRUYY7}.)
By the last author's classification \cite{Y},
geometric catenoids are divided into 
the following eight classes:
\begin{itemize}
\item G-catenoids of type TE, TP or TH
(which are special cases of W-catenoids with elliptic, parabolic or
hyperbolic monodromies, respectively), 
\item G-catenoids of type SE, SP or SH 
(exceptional W-catenoids of type II belong
to a special class of G-catenoids of type SE),
\item G-catenoids of type LE or LH,   
\end{itemize}
where $T,S$ and $L$ denote whether the geodesic
fixed by the 1-parameter group of isometries 
is time-like, space-like or light-like,
and $E,P$ and $H$ stand for elliptic, parabolic and hyperbolic, respectively.
We show that,
except for G-catenoids of type TE, TP and TH which are 
already analytically complete, 
the other five types of G-catenoids have
analytic extensions which are analytically complete.

To formulate the notion of analytic extension of analytic maps,
we define the concept 
\lq\lq double-cone manifold''. 
In fact, all analytic extensions of G-catenoids 
admit only cone-like singular points 
as their singular points, and can be parametrized
by analytic maps defined on double-cone manifolds.

In Section 1, we define analytic completeness. 
In Section 2, we define analytic maps on double-cone manifolds
and show their fundamental properties.
All analytic extensions of G-catenoids
can be interpreted as $\op{DC}$-immersions (i.e. double-cone immersions, 
that is immersions of $\op{DC}$-manifolds). 
In Section 3, we give a criterion for analytic completeness of
real analytic $\op{DC}$-immersions.
In Section 4, we prove that G-catenoids belonging to the above 
five types SE, SP, SH, LE or LH can be extended as 
real analytic $\op{DC}$-immersions from certain 
$\op{DC}$-manifolds of dimension~2.

The analytic completeness of maximal catenoids or their particular
analytic extensions in $\R^3_1$ is also discussed 
in Section 2 and  
Appendix A.

\section{Preliminaries}

Throughout this paper,
we fix a real analytic manifold $N^n$ of dimension $n(\ge 1)$.
We also fix a positive integer $m(\le n)$.
We begin by defining analytic completeness as follows:

\begin{Definition}
A non-empty subset $S$ of $N^n$
is said to be {\it analytically complete}
if it satisfies the following 
property: 
\begin{quote}
Any real analytic map $\Gamma : [0,1] \to N^n$ satisfying  
$\Gamma([0,\varepsilon)) \subset S$ for some $\varepsilon \in (0,1)$
satisfies $\Gamma([0,1]) \subset S$.
\end{quote}
\end{Definition}

Analytic completeness is the same concept as 
the \lq\lq arc-symmetric property\rq\rq\
given in \cite{arc}. 
For example:
 \begin{itemize}
 \item 
Any affine subspace $A^m$ in $\R^n$ is an analytically complete 
subset of $\R^n$ (cf. Example \ref{ex:Rn} 
and Proposition \ref{prop:first}). 
\item 
the interval $I:=[0,1]$ 
is not an analytically complete subset of $\R$.
In fact, let
$\Gamma:[0,1] \to \R$ be a
real analytic map defined by $G(t)=2t$
which maps $[0,1/2)$ to $I$ 
but $\Gamma(I)\not \subset I$.
\item Let $S = \{ (t, e^{-1/t}) \in \R^2 \mid 0 < t < \infty \}$. 
Then $S$ is an analytically complete subset of $\R^2$,
although $S$ has a $C^\infty$-extension
(cf. Example \ref{ex:Kok}).
 \end{itemize}

\begin{Definition}
A subset $S$ of $N^n$ is called {\it globally analytic} 
if there exists a positive integer $l$ and a real analytic map
$\tau:N^n \to \R^l$ such that the zero set
$$
Z(\tau):=\Big\{P\in N^n\,;\, \tau(P)=(0,\dots,0)\Big\}
$$
of $\tau$ coincides with $S$.  
\end{Definition}

\begin{Exa}\label{ex:Rn}
If $S$ is an $m$-dimensional linear subspace 
in $\R^n$, then $S$ is a globally analytic subset,
as there is a linear map $F:\R^n\to \R^{l}$ ($l\ge 1$)
such that $S$ is the kernel of $F$.
\end{Exa}

\begin{Exa}
For $m\ge 2$, the set 
\begin{align*}
\mc C^m_1:=\left\{(x_1,\dots,x_{m+1})\in \R^{m+1}\,;\, 
\sum_{j=1}^m x_j^2=x_{m+1}^2,\,\,\, x_{m+1}\ge 0\right\}
\end{align*}
is called the {\it $m$-dimensional standard cone},
which is a subset of
\begin{equation}\label{eq:cone}
\mc C_2^m:=\left\{(x_1,\dots,x_{m+1})\in \R^{m+1}\,;\, 
\sum_{j=1}^m x_j^2=x_{m+1}^2\right\}.
\end{equation}
This $\mc C_2^m$ is called
the {\it $m$-dimensional standard double-cone}
and is a globally analytic subset of $\R^{m+1}$.
Moreover, $\mc C^m_2$ is the image of the analytic map
\begin{equation}\label{eq:fdc}
\pi_{\op{DC}}^{m}:S^{m-1}\times \R \ni (\mb x,t)\mapsto (t \mb x,t)\in \R^{m+1},
\end{equation}
where $S^{m-1}$ is the unit sphere centered at the origin in $\R^{m}$.
We call the origin $\mb 0\in \mc C_2^m$ the 
{\it standard double-cone point}.
\end{Exa}

As pointed out in the introduction,
the subsets $\mc E,\,\,\mc P$  and $\mc H$ of $\R^3_1$
given by \eqref{eq:E0}, 
\eqref{eq:p2} and \eqref{eq:H} are globally analytic.
The following assertion holds:

\begin{Proposition}\label{prop:first}
Let $S$ be a subset of $N^n$.
If $S$ is a globally analytic subset of $N^n$, then it
 is analytically complete.
\end{Proposition}

\begin{proof}
Suppose that $S$ is globally analytic.
Then there exist a positive integer 
$l$ and a real analytic map $\tau:N^n\to \R^l$
such that $S=Z(\tau)$. Let $\Gamma:[0,1]\to N^n$ be a real analytic map
such that $\Gamma([0,\epsilon))\subset S(\subset Z(\tau))$ 
for some sufficiently small $\epsilon>0$.
Then the real analyticity of $\Gamma$ and $\tau$
implies $\Gamma([0,1])\subset S$.
So $S$ is analytically complete.
\end{proof}

\section{Double-cone manifolds ($\op{DC}$-manifolds)}

In this section, we consider real analytic maps
as well as smooth maps. So, by $C^r$-map, 
we mean a real analytic (resp. smooth) map if 
$r=\omega$ (resp. $r=\infty$).
In the previous section, we defined
the standard double-cone $\mc C^m_2$. 
Using this, we define double-cone 
manifolds (i.e. $\op{DC}$-manifolds):
Related to this are the concepts of \lq\lq orbifold" and 
\lq\lq conifold", but the first one requires a group action
and the second one is considered on complex manifolds
or on real manifolds 
modeled on the standard cone $\mc C^m_1$.
As far as the authors know, there is no similar definition  
of manifolds modeled on the standard double-cone $\mc C^m_2$.
We first define topological double-cone 
manifolds (i.e. $\op{DC}$-manifold) as follows:

\begin{Definition}\label{def:DC0}
Fix an integer $m(\ge 2)$.
A Hausdorff space $X$ with the second axiom of
countability is called a 
{\it topological $\op{DC}$-manifold of dimension $m$}
if each point of $X$ has a neighborhood
which is homeomorphic to an open subset of $\mc C^m_2$.
Moreover, $p\in X$ is called a {\it $\op{DC}$-point} if
$p$ corresponds to the origin of $\mc C^m_2$.
\end{Definition}

We let $\Sigma$ be the set of all $\op{DC}$-points of
a topological $\op{DC}$-manifold $X$ of dimension $m$.
Then, by definition, $X\setminus \Sigma$ 
is a topological manifold of dimension $m$ and
$\Sigma$ is a discrete subset of $X$.

\begin{Definition}\label{def:DC1}
Let $X$ be a topological $\op{DC}$-manifold of dimension $m$, and  
let $\varphi$ be a homeomorphism 
of an open subset $\hat U$ of $\mathcal C^m_2$
onto a connected open set $O(\subset X)$.
Since $\pi^m_{\op{DC}}$ (cf. \eqref{eq:fdc}) is a
continuous map,
\begin{equation*}
 U := (\pi^m_{\op{DC}})^{-1}(\hat U)
\end{equation*}
is an open subset of $S^{m-1}\times \R$. 
We consider
a continuous map
$\Phi:U\to O$ given by
$
 \Phi := \varphi \circ \pi^m_{\op{DC}},
$
and call the pair $(U, \Phi) $ an 
\emph{inverse $\op{DC}$-coordinate system} (or a \emph{parametrization system}).
We then call the pair $(O, \varphi^{-1}) $ a \textit{coordinate system}
associated with $(U,\Phi)$.
\end{Definition}

\begin{Definition}\label{def:DC2}
A \emph{differentiable $\op{DC}$-structure $\mathcal F$ of class $C^r$} 
on an $m$-dimensional topological $\op{DC}$-manifold $X$
is a collection of inverse coordinate systems
$$
\{(U_\lambda, \Phi_\lambda)  ; \lambda \in \Lambda \}
$$
satisfying the following properties: 
\begin{enumerate}
\item[(i)] By setting $O_\lambda:=\Phi_\lambda(U_\lambda)$,
the family 
$\{O_\lambda \}_{\lambda \in \Lambda}$ is an open covering of $X$. 
\item[(ii)] Whenever $O_\lambda \cap O_\mu \ne \varnothing$, 
there exists a $C^r$-diffeomorphism 
\begin{equation*}
 \Phi_{\mu \lambda} \colon \Phi_\lambda^{-1}(O_\mu) \cap U_\lambda 
\to \Phi_\mu^{-1}(O_\lambda) \cap U_\mu 
\end{equation*}
satisfying $\Phi_\mu \circ \Phi_{\mu \lambda} = \Phi_\lambda$
(see Figure \ref{fig:CD1}). 
 \item[(iii)] The collection $\mathcal F$ is maximal with respect to (ii).
\end{enumerate}
Moreover, 
an $m$-dimensional topological $\op{DC}$-manifold $X$
with a differentiable $\op{DC}$-structure of class $C^r$
is called a ($C^r$-differentiable) {\it $m$-dimensional
$\op{DC}$-manifold}. 
In this setting, each 
$(U_\lambda, \Phi_\lambda)$ $(\lambda \in \Lambda)$
is an inverse $\op{DC}$-coordinate system of $X$,
and $(O_\lambda, \phi_\lambda^{-1})$ is  
a local coordinate system   of $X$ associated with
$(U_\lambda, \Phi_\lambda)$.
\end{Definition}

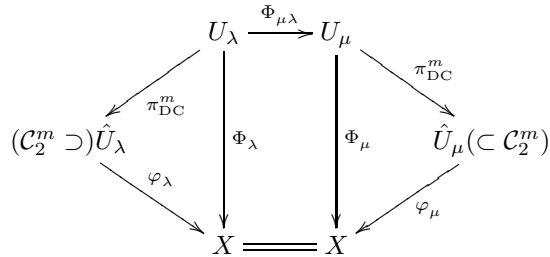
\begin{figure}[htb]%
 \begin{center}
\[
\xymatrix{
& U_\lambda \ar[dd]^{\Phi_{\lambda}^{}}
\ar[r]^{\Phi_{\mu\lambda}^{}} \ar[ld]^{\pi^m_{\op{DC}}} 
& U_\mu \ar[dd]^{\Phi_{\mu}^{}} \ar[rd]^{\pi^m_{\op{DC}}} &\\ 
(\mc C^m_2\supset) \hat U_\lambda
\ar[rd]^{\phi_{\lambda}^{}} 
	& & 
& \hat U_\mu (\subset \mc C^m_2) \ar[ld]^{\phi_{\mu}^{}} 
\\
& X \ar@{=}[r]& X&} 
\]
\end{center}
  \caption{The commutative diagram 
related to the maps $\Phi_\lambda$ and $\Phi_\mu$.
}%
\label{fig:CD1}
\end{figure}

Usual manifolds can be considered as 
$\op{DC}$-manifolds without $\op{DC}$-points:
We let $M^m$ be an $m$-dimensional manifold, and $(O,\psi)$ 
a local coordinate neighborhood of a point $q\in M^m$.
Since $S^{m-1}\times (0,\infty)$
is diffeomorphic to $\R^m\setminus \{\mb 0\}$, 
we may assume that 
$\psi(O)$
is a subset of 
$S^{m-1}\times (0,\infty)$.
Then
\begin{equation}\label{eq:checkU}
(\psi(O), \Phi)\qquad (\Phi:=\psi^{-1})
\end{equation}
gives a local inverse $\op{DC}$-coordinate system 
of $M^m$ as a $\op{DC}$-manifold.

The standard double-cone $\mc C_2^m$ 
is a typical example of a $C^r$-differentiable $\op{DC}$-manifold
with a single local inverse $\op{DC}$-coordinate system
$
\Phi:=\pi_{\op{DC}}^{m}:S^{m-1}\times \R  \to \mc C_2^m.
$ 

\begin{Definition}\label{def:mDC-map}
Let $X^m$ and $Y^n$ be two $C^r$-differentiable $\op{DC}$-manifolds whose 
$\op{DC}$-structures are
$\{(U_\alpha,\Phi_\alpha)\}_{\alpha\in A}$
and
$\{(V_\beta,\Psi_\beta)\}_{\beta\in B}$, respectively.
A map $f:X^m \to Y^n$ is called a 
$C^r$-differentiable
{\it $\op{DC}$-map}
if for each  pair of
indices $(\alpha,\beta)\in A\times B$, there exists
a $C^r$-map
$$
F_{\alpha,\beta}:
U_\alpha\cap \Phi_\alpha^{-1}\Big(f^{-1}\big(\Psi_\beta(V_\beta)\big)\Big)
\to V_\beta
$$
satisfying
$
\Psi_\beta\circ F_{\alpha,\beta}=f\circ \Phi_\alpha.
$
Moreover, if $f$ is bijective and $f^{-1}$ is
also a  $C^r$-differentiable $\op{DC}$-map,
then $f$ is called a
($C^r$-differentiable) {\it $\op{DC}$-diffeomorphism}.
\end{Definition}

\begin{Remark}
Let $(O,\psi)$ be a local coordinate system of
a $\op{DC}$-manifold $X^m$, then $\psi:O\to \psi(O)\,(\subset \mc C^m_2)$
gives a typical example of a $\op{DC}$-diffeomorphism.
\end{Remark}

We next define \lq\lq DC-immersions'' as follows:

\begin{Definition}\label{def:DCI}
Let $X^m$ be a $C^r$-differentiable $\op{DC}$-manifold,
and let $N^n$ be a $C^r$-manifold. 
A map $f:X^m\to N^n$ ($n> m$) is 
called a {\it $\op{DC}$-immersion} if,
for each $p\in X$, there exist
\begin{itemize}
\item  a local DC-coordinate system $(O,\psi)$ 
containing $p$ and 
\item a $C^r$-diffeomorphism $F$ 
defined on a neighborhood $\Omega_0$ of the origin $\mb 0$
of $\R^{n}$ 
onto a neighborhood $\Omega$ of $f(p)$ 
such that 
$$
F(a,\mb 0_{n-m-1}) =f\circ \psi^{-1}(a)
\qquad (a\in \psi(O)\subset \mc C^m_2),
$$
where $\mb 0_{n-m-1}$ is the zero vector of $\R^{n-m-1}$.
\end{itemize}
\end{Definition}	

By definition, we have the following:

\begin{Proposition}\label{prop:DC}
A $\op{DC}$-immersion is a $\op{DC}$-map.
\end{Proposition}

The following assertion is also obvious:

\begin{Proposition}
Let $M^m$ be a $C^r$-manifold,
and let $f:M^m\to N^n$ be a $C^r$-immersion
into a $C^r$-manifold $N^n$ of dimension $n(>m)$.
Then $f$ is a $\op{DC}$-immersion. 
\end{Proposition}

As a consequence, we have the following:

\begin{Corollary}\label{cor:Xm}
Let $X^m$ be a  $C^r$-differentiable $\op{DC}$-manifold
and $f:X^m\to N^n$ $(m<n)$ a 
$\op{DC}$-immersion. 
If $g:N^n\to L^l$ $(l\ge n)$
is a $C^r$-immersion into a $C^r$-manifold of dimension $l$,
then  $g\circ f$ is also a $\op{DC}$-immersion.
\end{Corollary}

\begin{Definition}
A subset $X$ of a $C^r$-manifold $N^{n}$ 
is said to be {\it admissible} if there exists a
discrete subset $\Gamma$ of $X$ such that $X \setminus \Gamma$
is an $m$-dimensional $C^r$-submanifold of $N^{n}$,
where $m<n$.
We call such $\Gamma$ a {\it pre-residual subset} of $X$.
We consider the intersection of all pre-residual subsets
of $X$, which gives the 
smallest pre-residual subset of $X$.
We denote it by  $\Sigma$, and call it  the {\it residual set} of $X$.
\end{Definition}

By definition, $X\setminus \Sigma$ is an
$m$-dimensional $C^r$-submanifold of $N^{n}$.

\begin{Definition}\label{def:X}
Let $X$ be an admissible subset of
a $C^r$-manifold $N^{n}$ ($n> m$), and
let $\Omega(\subset \R^{n})$ 
be an open subset so that
$$
\hat U:=(\mc C^m_2\times \{\mb 0_{n-m-1}\})\cap \Omega
$$
is non-empty.
A $C^r$-diffeomorphism $F$ from $\Omega$
into $N^{n}$ is called a $C^r$-differentiable
{\it extended parametrization} of $X$
if 
$\hat U$
satisfies the following properties;
\begin{itemize}
\item[(1)] $F(\hat U)$ is contained in $X$,
\item[(2)] if $\hat U$ contains the origin $\mb 0$,
then $F(\mb 0)$ belongs to the residual set $\Sigma$ of $X$, and
\item[(3)] $F$ gives a $C^r$-map from 
the $C^r$-submanifold $\hat U\setminus \{\mb 0\}$
of $\R^{m+1}\times\{\mb 0_{n-m-1}\}$ into 
the $C^r$-submanifold $X\setminus \Sigma$.
\end{itemize}
\end{Definition}

Regarding $\mc C^m_2$ as
an admissible subset of $\R^{m+1}$,
we prepare the following:

\begin{Lemma}\label{lem:new-key}
Let $\Omega_i$ $(i=1,2)$
be open subsets containing the origin of $\mc C^m_2$.
Suppose that $F:\Omega_1\to \R^{m+1}$ is a $C^r$-differentiable
extended parametrization of
$\mc C^m_2$
such that  $F(\Omega_1)=\Omega_2$.
Then the restriction $f:=F|_{\Omega_1\cap \mc C^m_2}$
is a $C^r$-differentiable $\op{DC}$-diffeomorphism
between $\Omega_1\cap \mc C^m_2$ and $\Omega_2\cap \mc C^m_2$.
\end{Lemma}

\begin{proof}
Since we can switch the roles of $F$ and $F^{-1}$, 
it is sufficient to show that $f$ is 
a $\op{DC}$-map at each point $p\in \mc C^m_2\cap \Omega_1$.
If $p$ is not the origin $\mb 0$ in $\R^m$,
the assertion is obvious.
So we may assume that  $p=\mb 0$.
By (2) of 
Definition \ref{def:X},
we have $f(p)=\mb 0$.
We set
$
U_i:=(\pi^m_{\op{DC}})^{-1}(\Omega_i)\,\, (i=1,2).
$
Since $\mb 0$ belongs to $\Omega_1$ and $\Omega_2$,
the subset $S^{m-1}\times \{0\}$ is contained in $U_1$
and $U_2$. We set
$
G:=F\circ \pi_{\op{DC}}^{m},
$
which is a $C^r$-map from $U_1$ to $\R^{m+1}$.
Using the central projection
$
\pi_C:\mc C^m_2\setminus \{\mb 0\}
\ni(\mb{x},t)\mapsto  {\mb{x}}/{t}\in S^{m-1},
$
we set
$$
g:U_1\setminus (S^{m-1}\times \{0\})
\ni (\mb{x},t) \mapsto \pi_{C}\circ G(\mb{x},t)\in S^{m-1}.
$$
If we set
$$
\big(\lambda(\mb{x},t),\tau(\mb{x},t)\big):=
G(\mb{x},t)\in \mc C^m_2(\subset \R^{m+1}),
$$
then
$$
\alpha(\mb{x},t):=\frac{\lambda(\mb{x},t)}{t}, \quad
\beta(\mb{x},t)=\frac{\tau(\mb{x},t)}t
$$
are both $C^r$-differentiable maps
at $t=0$  (cf. \cite[Appendix A]{SUY2}).
Since
$t\mapsto G(\mb{x},t)$ is a regular curve,
neither
$$
\alpha(\mb{x},0)=\frac{\partial \lambda(\mb{x},0)}{\partial t}
\,\,\,\text{nor} \,\,\,
\beta(\mb{x},0)=\frac{\partial \tau(\mb{x},0)}{\partial t}
$$
vanishes at $t=0$.
Moreover, since
$\lambda(\mb{x},t)/\tau(\mb{x},t)$ is a unit vector for $t \ne 0$,
the order of the map $t\mapsto \lambda(\mb{x},t)$ at $t=0$
coincides with that of
$\tau(\mb{x},t)$.
In particular, $\partial \alpha(\mb{x},0)/\partial t$ and
$\partial \beta(\mb{x},0)/\partial t$
are both non-zero.
Thus, the map 
$$
\Phi:U_1\ni (\mb x,t)\mapsto \left(\frac{\lambda(\mb x,t)}{\tau(\mb x,t)},\tau(\mb x,t)\right)\in U_2
$$
is a $C^r$-map, and satisfies $\pi^m_{\op{DC}}\circ \Phi=G$,
which implies that
$f$ is $C^r$-differentiable on
$\pi^m_{\op{DC}}(U_1)$, proving the assertion.
\end{proof}

\begin{Definition}\label{eq:def-N}
Let $N^{n}$ be an $n$-dimensional $C^r$-manifold
and $X$ an admissible subset of $N^{n}$.
Then $X$ is said to be a {\it $\op{DC}$-submanifold} of $N^{n}$
if $X$ has an $m$-dimensional $\op{DC}$-structure ($m<n$)
such that, for any
extended parametrization 
$F:\Omega(\subset \R^n)\to N^{n}$
of $X$ (cf. Definition 
\ref{def:X}), the pair 
$$
\Big((\hat \pi_{\op{DC}}^m)^{-1}(\Omega),
F\circ \hat \pi^m_{\op{DC}}\Big)
$$
gives a local inverse $\op{DC}$-coordinate system of $X$,
where 
$$
\hat \pi_{\op{DC}}^m:S^{m-1}\times \R
\ni a\mapsto (\pi_{\op{DC}}^m(a),\mb 0)\in \mc C^m_2\times \R^{n-m-1} \qquad
 (a\in S^{m-1}\times \R).
$$
\end{Definition}

By definition, the following assertion is obvious:

\begin{Proposition}
The inclusion map of a DC-submanifold
is a  DC-immersion.
\end{Proposition}

We next prepare the following:

\begin{Lemma}\label{lem:ADC}
Let $N^{n}$ be an $n$-dimensional $C^r$-manifold
and $X$ an admissible subset of $N^{n}$.
Suppose that, for each $p\in X$, there exist 
\begin{enumerate}
\item an open subset $\Omega_p$ of $\R^{n}$, and
\item a $C^r$-differentiable extended parametrization 
$F_p:\Omega_p\to N^{n}$ of $X$
such that $p\in F_p(\Omega_p)$.
\end{enumerate}
Then, $X$ has  the structure of a
$C^r$-differentiable $\op{DC}$-submanifold of $N^{n}$
such that the residual set is the set of DC-points
and the inclusion map $X \hookrightarrow N^{n}$
is a $\op{DC}$-immersion.
\end{Lemma}

\begin{proof}
We take $p,q\in X$ such that $O_{p,q}:=F_p(\Omega_p)\cap F_q(\Omega_q)$
is non-empty. Then $G:=F_q^{-1}\circ F_p$ is an
extended parametrization of 
$\hat {\mc C}^m_2$ (cf. Definition \ref{def:X})
which maps  
$F_p^{-1}(O_{p,q})$ onto $F_q^{-1}(O_{p,q})$.
By Lemma \ref{lem:new-key},
the restriction of $G$
is a $C^r$-differentiable $\op{DC}$-diffeomorphism
between $\Omega_p\cap G^{-1}(\Omega_q)\cap \hat{\mc C}^m_2$ and 
$\Omega_q\cap G(\Omega_p)\cap \hat{\mc C}^m_2$.
Thus, by setting $U_p:=(\hat \pi_{\op{DC}}^m)^{-1}(\Omega_p)\cap 
\hat{\mc C}^m_2$, 
the family
$
\left\{(U_p, F_p\circ 
\hat \pi_{\op{DC}}^m|_{U_p})\right\}_{p\in X}
$
gives the structure of a $C^r$-differentiable
$m$-dimensional $\op{DC}$-manifold on $X$,
which can be considered as a $\op{DC}$-submanifold of $N^{n}$.
It is obvious that 
the inclusion map $X \hookrightarrow N^{n}$
is a $C^r$-differentiable $\op{DC}$-immersion.
\end{proof}

We then prove the following:

\begin{Proposition}\label{prop:Xstr}
Let $X$ be an admissible subset 
$($with the residual set $\Sigma)$
of a $C^r$-manifold $N^{n}$ of dimension $n$.
Suppose that for each $p\in \Sigma$, there exist 
\begin{enumerate}
\item an open subset $\Omega_p$ of $\R^{n}$, and
\item a $C^r$-differentiable 
extended parametrization $F:\Omega_p\to N^{n}$ of $X$
such that $p\in F_p(\Omega_p)$.
\end{enumerate}
Then $X$ has the structure of a $C^r$-differentiable  
$\op{DC}$-submanifold 
of $N^{n}$ such that the residual set is the set of $\op{DC}$-points
and the inclusion map $X \hookrightarrow N^{n}$
is a $\op{DC}$-immersion.
\end{Proposition}

\begin{proof}
We let $q\in X\setminus \Sigma$. 
Since $X\setminus \Sigma$ is a $C^r$-submanifold of
$N^{n}$, we can take a sufficiently small local coordinate 
neighborhood $(W,\psi) (\subset X\setminus \Sigma)$ 
of $q$ satisfying $W\subset X\setminus \Sigma$.
Since $\psi(W)$ is an open subset of $\R^m$,
this can be considered as an open subset of
$\hat{\mc C}^m_2$.
Moreover, without loss of generality, we may assume that
there exist
\begin{itemize}
\item an open subset $\Omega_0(\subset \R^{n})$
satisfying $\Omega_0\cap \hat{\mc C}^m_2=\psi(W)$,
\item an open subset $\Omega$ of $N^{m+1}$
such that $\Omega\cap X=W$, and
\item an extended parametrization $F:\Omega_0\to \Omega$ of $X$
\end{itemize}
such that $q\in F(\Omega_0)$.
So the assertion follows from Lemma~\ref{lem:ADC}.
\end{proof}

Here, we show that
the image $\mc E$ and the extended images 
$\mc P$ and $\mc H$ of geometric catenoids in $\R^3_1$
given in the introduction can be considered as 
$\op{DC}$-submanifolds in $\R^3_1$
as follows:

\begin{Exa}\label{ex:E0}
The subset $\mc E$ of $\R^3_1$ given in \eqref{eq:E0}
coincides with
the image of  the elliptic $G$-catenoid $f_E$,
which
can be characterized as 
the zero set of the function
$$
\tau_E:=x^2+y^2-\sinh^2 t.
$$
The exterior derivative of $\tau_E$  
vanishes only at the origin $\mb 0$ of $\R^3_{1}$. 
So $\mc E \setminus \{\mb 0\}$ has the structure of 
a (usual) real analytic $2$-manifold and $\mc E$
is an admissible subset of $\R^3$.
Moreover, the map defined by
$$
F_E:\R^3\ni (x_1,x_2,x_3)\to (\sinh^{-1} x_3, x_1,x_2)\in \R^3_1
$$
is an extended parametrization of 
$\mc E$ (cf.~Definition~\ref{def:X}),
where $\sinh^{-1}$ is the inverse function of $\sinh$.
By Proposition \ref{prop:Xstr},
$\mc E$ is a real analytic 
$\op{DC}$-submanifold of $\R^3$.
\end{Exa}

\begin{Exa}\label{ex:P0}
We consider the subset $\mc P$ given by \eqref{eq:p2}
of $\R^3_1$, which is the zero-level set of the real analytic function 
$
\tau_P= 12 (x^2 - t^2) - (x + t)^4 + 12 y^2.
$
The exterior derivative of $\tau_P$
vanishes only at the origin $\mb 0$ of $\R^3_{1}$. 
So $\mc P \setminus \{\mb 0\}$ has the structure of 
a (usual) real analytic $2$-manifold and $\mc P$
is an admissible subset of $\R^3$.
If we set
$$
x_1+x_3:=x+t,\qquad x_1-x_3:=12(x-t)-(x+t)^3,\quad x_2:=2\sqrt{3}y,
$$
then
$\tau_P=0$ is equivalent to the relation 
$(x_1)^2+(x_2)^2=(x_3)^2$.
Then the map 
$F_P:\R^3 \to
\R^3_1$
defined by
$$
F_P(x_1,x_2,x_3): 
=\frac1{24}
\Big(11 x_1+13 x_3-(x_1+x_3)^3,13 x_1+11 x_3+(x_1+x_3)^3,4\sqrt{3}x_2\Big)
$$
is an extended parametrization of
$\mc P$ (cf. Definition \ref{def:X}), and
by Proposition \ref{prop:Xstr},
$\mc P$ is a real analytic 
$\op{DC}$-submanifold of $\R^3_1$.
\end{Exa}

\begin{Exa}\label{ex:H0}
We consider the subset $\mathcal H$ given by \eqref{eq:H0},
which is the zero-level 
set of the real analytic function $\tau_H= \sin^2 x + y^2 -t^2$.     
It can be easily checked that $\mc H$
is an admissible subset of $\R^3_1$
whose residual set is
$\Sigma=\{(0,k\pi,0)\,;\, k\in \Z\}$.
We set
$$
\Omega_0:=\{(x_1,x_2,x_3)\in \R^3\,;\, |x_1|< 1\}
$$
and, for every $k \in \mathbf Z$,
consider the map 
\begin{equation*}
 F_k \colon \Omega_0 \ni (x_1,x_2,x_3) \mapsto 
(x_3,k\pi+\arcsin x_1, x_2) \in \R^3_1,
\end{equation*}
which is an extended
parametrization of $\mc H$ (cf. Definition \ref{def:X})
such that $(0,k\pi,0)\in F_k(\Omega_0)$.
By Proposition \ref{prop:Xstr},
$\mc H$ is a real analytic
$\op{DC}$-submanifold of $\R^3$.
\end{Exa}

\begin{Definition} \label{def:Mdim}
Let $X^m$ be a connected $m$-dimensional $\op{DC}$-manifold, and
$N^n$ be an $n$-dimensional real analytic manifold.
A $\op{DC}$-map
$f:X^m\to N^n$ is said to be
{\it $m$-dimensional} if
there exists an open dense subset $O$
of $X^m$ consisting of non-$\op{DC}$-points of $X^m$
such that $f$ is an immersion on $O$.
\end{Definition}

By definition, DC-immersion from 
a connected $m$-dimensional $\op{DC}$-manifold
is a special case of 
$m$-dimensional DC-map.
From here we restrict our attention  to the real analytic case.
The following assertion holds:

\begin{Proposition}\label{prop:Long}

Let $X^m$ be a connected $\op{DC}$-manifold, and
$f:X^m\to N^n$ $(m\le n)$
be a real analytic $\op{DC}$-map
$($Definition \ref{def:mDC-map}$)$.
If there exists a non-$\op{DC}$-point $p\in X^m$ at which
$(df)_p$ is injective,
then $f$ is an $m$-dimensional DC-map.
\end{Proposition}

\begin{proof}
We denote by $\Sigma$ the set of $\op{DC}$-points of $X^{m}$.
Then $X^m\setminus \Sigma$ is a union of
connected real analytic manifolds $\{C_\lambda\}_{\lambda\in \Lambda}$.
We can take
$\mu_0\in \Lambda$ so that 
$p\in {C_{\mu_0}}$.
Since $(df)_p$ is injective, 
$f|_{C_{\mu_0}}$ gives an immersion on a dense subset of $C_{\mu_0}$.
It is sufficient to show that
each $
f|_{C_\lambda}$ ($\lambda\in \Lambda$) 
gives an immersion on a dense subset of $C_{\lambda}$.
We suppose that it fails.
Then there exists $\lambda_0$ such that
the rank of $(df)_p$ is less than $m$ for each $p\in C_{\lambda_0}$.
Since $X^m$ is connected, there exists a continuous map $\gamma:[0,1]\to
X^m$ such that $\gamma(0)=p$ and 
$\gamma(1)\in C_{\lambda_0}$. 
Since $\Sigma$ is discrete,
$\gamma$ passes through only finitely many connected components
of $X\setminus \Sigma$. 
So we may assume that $C_{\mu_0}$ is 
adjacent to $C_{\lambda_0}$, that is, there exists
a point $q\in \Sigma$ which connects to
$C_{\lambda_0}$ with $C_{\mu_0}$,
without loss of generality.
We let $(U,\Phi)$ be a local inverse $\op{DC}$-coordinate system
such that $q\in \Phi(U)$ and $U$ is connected.
Then, there exists $x_0\in U$ such that $\Phi(x_0)=q$.
We can take open subsets $V,W(\subset U)$
such that $\Phi(V)\subset C_{\lambda_0}$ and
$\Phi(W)\subset C_{\mu_0}$.
By our choice of $\lambda_0$,
the rank of the differential of $f\circ \Phi$ is less than $m$ on  $V$.
Then, by the real analyticity of $f$,
the rank of the differential of $f\circ \Phi$
must be less than $m$ on $U$ because of the connectedness of $U$.
Since $\Phi(W)\subset C_{\mu_0}$,
this contradicts the fact that
$f|_{C_{\mu_0}}$ gives an immersion on a
dense subset of $C_{\mu_0}$.
\end{proof}

\begin{Definition}\label{def:r-A6}
Let $X^m_i$ be connected $\op{DC}$-manifolds of dimension $m$, 
and
$f_i:X^m_i\to N^n$ ($i=1,2$)
be two real analytic $\op{DC}$-analytic maps. 
Then $f_2$ is called an {\it analytic extension} 
of $f_1$ if there exists
a $\op{DC}$-analytic map $\phi:X^m_1\to X^m_2$
such that $f_1=f_2\circ \phi$
and $f_1(X^m_1)$ 
is a proper subset of $f_2(X^m_2)$. 
\end{Definition}

\begin{Proposition}\label{prop;Aext}
Let $f:X^m\to N^n$ be a real analytic $m$-dimensional $\op{DC}$-map. 
If $f(X^m)$ is analytically complete, then $f$
does not admit any analytic extension.
\end{Proposition}

To prove this, we prepare the following:

\begin{Lemma}\label{lemma;Aext}
Let $S$ be an analytically complete subset of $N^n$, and let $g:M^m\to N^n$ be an
$m$-dimensional real analytic map 
defined on a  connected $m$-dimensional real analytic
manifold $M^m$
$($cf. Definition \ref{def:Mdim}$)$.
If there exists a non-empty open subset $W$ of $M^m$
satisfying $g(W)\subset S$, then
$g(M^m)$ is also contained in $S$.
\end{Lemma}

\begin{proof}
Since $g$ is $m$-dimensional,
we can find a regular point $p\in U$ of $g$.
We fix an arbitrary  point $q\in M^m$.
Since $M^m$ is connected, there exists a real
analytic map $\sigma:[0,1]\to M^m$ such that
$\sigma(0)=p$ and $\sigma(1)=q$
(cf. \cite[the footnote on page 402]{HP}).
Since $W(\subset M^m)$ is an open subset,
there exists $\epsilon>0$ such that $\sigma([0,\epsilon))$
is contained in $W$.
Then $g\circ \sigma([0,\epsilon))\subset S$, and
the analytic completeness of
$S$ implies that $g\circ \sigma(1)$ also lies in
$S$.
\end{proof}

\begin{proof}[Proof of Proposition \ref{prop;Aext}]
By way of contradiction,
we suppose that $g:Y^m\to N^n$ is a  
$\op{DC}$-map giving an analytic extension of $f$. 
By definition of analytic extension,
there exists
a $\op{DC}$-analytic map $\phi:X^m_1\to X^m_2$
such that $f=g\circ \phi$ on $X^m_1$.
Since $f$ is $m$-dimensional,
applying Proposition \ref{prop:Long},
one can easily observe that
$\phi$ and
$g$ are both $m$-dimensional DC-maps.
Since $f$ is $m$-dimensional,
there exists a non-$\op{DC}$-point $p_0$
and a neighborhood $W$ of $p_0$
so that $f|_W$ is an immersion.
Since $\Sigma$ is discrete,
we may assume that $\phi(W)$
does not contain any $\op{DC}$-points. In particular, $\phi|_W$ is an immersion 
between the same dimensional manifolds,
and so, we may assume that $\phi$ is a $C^\omega$ 
diffeomorphism 
between $W$ and $\phi(W)$.

Since $g$ is an analytic extension of $f$,
we can find a point $q\in Y^m$ such that
$g(q)\not\in f(X^m)$.
Since $Y^m$ is connected, there exists a continuous map
$\gamma:[0,1]\to Y^m$ such that $\gamma(0)=p_0$ and $\gamma(1)=q$. 
Since the set of $\op{DC}$-points $\Sigma$ of $Y^m$ are discrete, 
the map $\gamma$ passes through only finitely many
$\op{DC}$-points $y_1,\ldots,y_k$. 
We let $(U_j,\Phi_j)$ ($j=1,\ldots,k$)
be a local inverse $\op{DC}$-coordinate system
so that $y_j\in\Phi_j(U_j)$ and $U_j$ is connected.
By changing the order of $y_1,\ldots,y_k$ if necessary,
there exist $k$ points $t_1,\cdots,t_k$
such that $\gamma(t_j)=y_j$
($j=1,\ldots,k$) and
$$
0=t_0<t_1<\cdots <t_k<t_{k+1}=1.
$$
Then there exist connected components
$C_1,\ldots,C_{k+1}$ of $Y^m\setminus \Sigma$ 
such that
$$
\gamma([t_{j-1},t_{j}])\subset C_j \qquad (j=1,\ldots,k+1).
$$
By Proposition \ref{prop:Long}, each $g|_{C_j}$
($j=1,\ldots,k+1$) is an $m$-dimensional analytic map.
By Lemma \ref{lemma;Aext},
the fact $g(\phi(W))\subset f(X^m)$
implies that $g(C_1)\subset f(X^m)$.
Since $g(q)\not\in f(X^m)$,
we have
$g(C_{k+1})\not\subset f(M^m)$.
So we can find an index $j\in \{1,\ldots,k\}$ such that
$g(C_j)\subset f(X^m)$ and $g(C_{j+1})\not \subset f(X^m)$.
Taking a sufficiently small $\epsilon>0$
so that
$$
\gamma([t_j-\epsilon,t_j+\epsilon])\subset \Phi(U_j),
$$
we can choose
$x_0,z_0\in U_j$
so that
$\Phi_j(x_0)=\gamma(t_{j}-\epsilon)$ and
$\Phi_j(z_0)=\gamma(t_{j}+\epsilon)$.
Moreover, we can also choose 
open subsets $V,W$ of $U_j$
satisfying
$$
x_0\in V, \quad z_0\in W,\quad
\pi_{\op{DC}}^m(V)\subset C_j,\quad
\pi_{\op{DC}}^m(W)\subset C_{j+1}.
$$
Since $h:=g\circ \pi_{\op{DC}}^m$ is real analytic on
the connected set $U_j$,
the fact $h(V)\subset f(X^m)$ 
and Lemma \ref{lemma;Aext} imply that
$h(U_j)\subset f(X^m)$.
In particular, $g(W) \subset f(X^m)$
holds.
Since $\pi_{\op{DC}}^m(W)$ is an open subset of $C_{j+1}$,
by Lemma \ref{lemma;Aext}, we can conclude
$g(C_{j+1})\subset f(X^m)$, a contradiction.
\end{proof}

For the analytic extension of G-catenoids and W-catenoids 
in $\R^3_1$, the following can be shown.

\begin{Theorem}\label{Cor:EPHK}
The image of the elliptic G-catenoid $f_E$ in $\R^3_1$
$($given in the introduction$)$
is analytically complete.
On the other hand, the image of the
other G-catenoids $f_P$ and $f_H$ 
and the W-catenoid $f_K$ in $\R^3_1$
admit analytic extensions, which are 
analytically complete.
Moreover, after taking the analytic completions, 
all of the images of
these catenoids are $\op{DC}$-submanifolds
of $\R^3_1$.
\end{Theorem}

\begin{proof}
By the classification of 
G-catenoids and W-catenoids in $\R^3_1$
(cf. Appendix A), we know that
G-catenoids are $f_E,f_P,f_H$ and
W-catenoids are $f_E,f_K$. 
The image of $f_E$ is the set $\mc E$ in Example \ref{ex:E0},
which is globally analytic. So $\mc E$ is analytically
complete.

On the other hand, 
as mentioned in the introduction,
the image of $f_K$
is a proper subset of the 
entire (zero-mean curvature) 
graph of the function
$
t=y \tanh x.
$
Since, the entire graph
is globally analytic
(cf. Proposition \ref{prop:first}),
it gives an analytic extension of $f_K$, which is
analytically complete.

As shown in  Example \ref{ex:P0},
the set
$\mc P$
given in 
\eqref{eq:p2}
has the structure of a $\op{DC}$-submanifold
whose inclusion map can be considered as an analytic extension 
of the
map $f_P$ given in \eqref{eq:p-catenoid}.
Since $\mc P$ is globally analytic, it is analytically complete
(cf. Proposition \ref{prop:first}).
Similarly, the set 
$\mc H$
given in 
\eqref{eq:H0}
has the structure of $\op{DC}$-submanifold
(cf. Example \ref{ex:H0}),
whose inclusion map can be
considered as an analytic extension of the
map $f_H$ given in  \eqref{eq:h-catenoid}.
Since $\mc H$ is globally analytic, it is analytically complete
(cf. Proposition \ref{prop:first}).
\end{proof}

\section{Analytic completeness for the images 
of real analytic $\op{DC}$-maps}\label{sec:ap}

In this section, we give a criterion which we can apply 
to show the analytic completeness of G-catenoids in Section 4.  
Throughout this section, we fix a real analytic
 $m$-dimensional DC-manifold $X^m$.
We first prepare several lemmas:

\begin{Lemma}\label{lem:X}
Let $X^m$ be a real analytic $\op{DC}$-manifold 
and $f:X^m\to N^n$ a $\op{DC}$-immersion.
Then for each $p \in X^m$, there exists a triple $(W,\Omega,\tau)$
consisting of
\begin{itemize}
\item a relatively compact neighborhood $W(\subset X^m)$
of $p$,
\item  a neighborhood $\Omega(\subset N^{n})$ 
of $f(p)$,
and 
\item a real analytic function
$\tau:\overline{\Omega}\to \R^{n-m}$
\end{itemize}
such that $Z(\tau)=f(\overline{W})$ and
$f|_{\overline{W}}$ is a homeomorphism between
$\overline{W}$ and $f(\overline{W})$,
where $Z(\tau)$ is the zero set of the function $\tau$.
\end{Lemma}

\begin{proof}
We fix $p\in X$ arbitrarily.
Then there exist
\begin{itemize}
\item an open subset $\Omega$ of $\R^{n}$,
\item a real analytic diffeomorphism $F:\Omega\to N^{n}$, and
\item 
a local inverse $\op{DC}$-coordinate system $(U,\Phi)$
satisfying $p\in \Phi(U)$
\end{itemize}
such that $\Omega\cap {\hat C}^m_2=\hat \pi_{\op{DC}}^m(U\times \{\mb 0_{n-m-1}\})$
and 
$
\Phi=F\circ \hat \pi_{\op{DC}}^m|_{U\times \{\mb 0_{n-m-1}\}}
$
holds on $U$ by identifying $U$ with $U\times \{\mb 0_{n-m-1}\}$
(cf. Definition \ref{def:DCI}).
Since
$$
\hat {\mc C}^m_2=\left\{(x_1,\ldots,x_n)\in \R^n\,;\, 
\sum_{i=1}^m x_j^2=x_{m+1}^2,\,\,
x_{m+2}=\cdots=x_{n}
\right\},
$$
the map
$$
\tau:=(h_0\circ F^{-1},x_{m+2}\circ F^{-1},\ldots,x_n\circ F^{-1}) \qquad
\left(h_0:=-x_{m+1}^2+\sum_{i=1}^m x_j^2\right)
$$
gives the desired  real analytic function.
In fact, we can take a relatively compact neighborhood $W$ of $p$
so that $f(W)\subset \Omega$.
Moreover, shrinking $\Omega$ if necessary,
we may assume that $F$ is defined on
$\overline{\Omega}$ such that 
$f(\overline{W})=\overline{\Omega}\cap Z(\tau)$. 
Since a bijective continuous map from a compact set
to a Hausdorff space is a homeomorphism,
$f|_{\overline{W}}$ is a homeomorphism between
$\overline{W}$ and $f(\overline{W})$.
\end{proof}

\begin{Lemma}\label{prop:X}
Let $f:X^m\to N^{n}$ be a $\op{DC}$-immersion.
Suppose that $I$ is an open interval of  $\R$ and that 
$\Gamma:I\to N^{n}$ is a real analytic map satisfying  
$\Gamma(I) \subset f(X^m)$.
Then there exists a continuous map 
$\gamma:J\to X^m$ 
defined on
an open subinterval $J$ of $I$
such that
$\Gamma=f\circ \gamma$ on $J$.
\end{Lemma}

\begin{proof}
We let $\{(U_p,\Omega_p,\tau_p)\}_{p\in X^m}$
be a family of triples determined by  the previous lemma.
Since $X^m$  
satisfies the second axiom of countability,
there exists a sequence $\{p_j\}_{j=1}^\infty$ 
of points in $X^m$ 
such that $\{U_{p_j}\}_{j=1}^\infty$ is
an open covering of $X^m$. 
We set
$$
W_j:=\Gamma^{-1}(\Omega_{p_j})\,\, (\subset I)
$$
and consider the map
$
\lambda_j:=\tau_{p_j}\circ \Gamma:W_j\to \R^l\,\, (j=1,2,\dots). 
$
We denote by $Z_j$ the zero set of $\lambda_j$
on $W_j$. Since $\Gamma(I) \subset f(X^m)$
and $f(X^m)\subset \bigcup_{j=1}^\infty Z(\tau_{p_j})$,
we have
$
I=\bigcup_{j=1}^\infty Z_j.
$
By Baire's category theorem, there exists a number $j_0$
such that $Z_{j_0}$ has an interior point.
So there exists an open subinterval $J$ of $I$
such that $J\subset Z_{j_0}$, that is,
$\gamma(J)\subset Z(\tau_{p_{j_0}})=f(U_{p_{j_0}})$.
Since $f|_{U_{p_{j_0}}}$ is a homeomorphism,
we can set
$
\phi:=(f|_{U_{p_{j_0}}})^{-1}\circ \Gamma,
$
and $f\circ \phi=\Gamma$ holds on $J$.
\end{proof}

\begin{Lemma}\label{prop:Y}
Let $f:X^m\to N^n$ be a $\op{DC}$-immersion.
Suppose that $I$ is an open interval containing 
$0\in \R$,
and
$\gamma_{i}:I\to X^m$ $(i=1,2)$
are two real analytic maps so that 
$f\circ \gamma_1(t)=f\circ \gamma_2(t)$
for all $t\in I$.
If $\gamma_1(0)=\gamma_2(0)$,
then $\gamma_1$ coincides with $\gamma_2$
as maps on the interval $I$.
\end{Lemma}

\begin{proof}
We consider the subset 
$
A:=\{t\in I\,;\, \gamma_1(t)=\gamma_2(t)\}.
$
Since $\gamma_1(0)=\gamma_2(0)$, we have $0\in A$, that is,
$A$ is a non-empty set.
Obviously, $A$ is a closed subset.
We fix $t_0\in A$ arbitrarily.
By Lemma \ref{lem:X},
there exist a neighborhood $O(\subset X^m)$
of $\gamma_j(t_0)$, a neighborhood $\Omega(\subset N^n)$ 
of 
$$
p_0:=f\circ \gamma_1(t_0)=f\circ \gamma_2(t_0)
\qquad (j=1,2)
$$
and a real analytic function
$\tau:\Omega\to \R^l$ ($l\ge 1$) 
such that $Z(\tau)=f(O)$ and
$f|_{O}$ gives a homeomorphism between
$O$ and $f(O)$.
By the continuity of $\gamma_j$ ($j=1,2$), 
there exists $\delta(>0)$ such that
$\gamma_j(t)\in O$ for $|t-t_0|<\delta$ for $j=1,2$.
By Lemma \ref{lem:X},
the inverse map $g$ of $f|_{O}$ is defined
and $\gamma_1(t)=g\circ f\circ \gamma_2(t)$
holds when $t\in (t_0-\delta,t_0)$.
Then $\gamma_1(t)=g\circ f\circ \gamma_2(t)$
holds for $t\in (t_0-\delta,t_0+\delta)$,
because $\gamma_1$ and $\gamma_2$ are real analytic.
So $(t_0-\delta,t_0+\delta)$ is a subset of $A$, and
$A$ is an open subset of $I$.
Since $I$ is connected, we have $A=I$, proving the
assertion.
\end{proof}

We prepare one more lemma:

\begin{Lemma}\label{prop:Z}
Let $f:X^m\to N^n$ be a $\op{DC}$-immersion, and let
$\Gamma:[a,b]\to N^n$  $(a<b)$ be a real analytic map.
If there exist $c\in (a,b)$, a continuous map 
$\gamma:[a,c)\to X^m$ 
and a sequence $\{t_n\}$ on $[a,c)$
converging to $c$
such that
$\Gamma=f\circ \gamma$ on $[a,c)$
and $(p:=)\dy\lim_{n\to \infty}\gamma(t_n)$ exists,
then there exists $0<\epsilon(\le b-c)$ 
and a continuous map $\tilde \gamma:[a,c+\epsilon)\to X^m$
such that $f\circ \tilde \gamma=\Gamma$
and $\tilde \gamma(t)=\gamma(t)$ for $t\in [a,c)$.
\end{Lemma}

\begin{proof}
By the continuity of $f$, we have
$$
f(p) =
f(\lim_{k\to \infty}\gamma(t_k))=\lim_{k\to \infty}f(\gamma(t_k))
=\lim_{k\to \infty}\Gamma(t_k)=\Gamma(c). 
$$
We then take a triple $(U,\Omega,\tau)$ at $p$
as in Lemma \ref{lem:X} so that
$\overline{U}$ is compact
and $f(\overline{U})$ coincides with 
the zero set of the $\R^{n-m}$-valued function $\tau$.
By the continuity of $\Gamma$, there exists
a sufficiently small $\epsilon(>0)$ such that
$\Gamma(c-\epsilon, c+\epsilon)\subset \Omega$.
Since 
$$
\Gamma ([a, c))=f\circ \gamma([a,c))\subset f(X)=Z(\tau),
$$
$\tau\circ \Gamma(t)$ is equal to zero  for $t\in [a,c)$,
by the real analyticity  of $\tau\circ \Gamma$,
we have $\tau\circ \Gamma(t)=0$ for
$t\in (c-\epsilon, c+\epsilon)$, which implies 
$\Gamma((c-\epsilon,c+\epsilon))\subset f(\overline{U})$.
Since $f|_{\overline U}$ is a homeomorphism (cf. Lemma \ref{lem:X}),
$
(f|_{\overline{U}})^{-1}\circ \Gamma(t)\,\, (t\in (c-\epsilon,c+\epsilon))
$ 
is a continuous map which coincides with $\gamma$
on $(c-\epsilon,c)$, and so $\gamma$ can be
extended to $[a,c+\epsilon)$.
\end{proof}

To define \lq\lq arc-properness", we give the following definition:

\begin{Def}
Let $X$ be a topological space,
and let $f:X\rightarrow N^n$ 
be a continuous map.
A continuous map 
$\sigma:[0,1)\to X$ is said to be {\it $(C^r,f)$-extendable}
if $f\circ \sigma(t)$ can be 
extended to a $C^r$-map defined on $[0,1]$,
where $r$ is a non-negative 
integer or $\infty$ or $r=\omega$.
\end{Def}

We now define the concept \lq\lq arc-properness" as follows:

\begin{Def}\label{def:ap}
Let $X$ be a locally compact Hausdorff space satisfying the
second axiom of countability (as a consequence,
$X$ is metrizable, affecting convergence of sequences of points on $X$). 
We let $f:X\rightarrow N^n$ 
be a continuous map.
We fix $r$ as a non-negative 
integer or $\infty$ or $r=\omega$. 
Then $f$ is called \textit{$C^r$-arc-proper}
if 
for each $(C^r,f)$-extendable
continuous map $\sigma:[0,1) \rightarrow X$, 
there exists a sequence $\{t_k\}_{k=1}^\infty$ 
on $[0,1)$
converging to $1$ 
such that
$
\dy\lim_{k\to \infty}\sigma(t_k)
$
exists.
\end{Def}

\begin{Prop}\label{prop:P}
Properness implies $C^r$-arc-properness
for $r\in \{0,1,\dots\}\cup \{\infty, \omega\}$. 
\end{Prop}

\begin{proof}
Let 
$f:X\rightarrow N^n$ be a proper 
continuous map. We fix a 
$(C^r,f)$-extendable
continuous map $\sigma:[0,1) \rightarrow X$. 
Then $f\circ \sigma(t)$ can be 
extended to a $C^r$-map $\Gamma:[0,1]\to N^n$.
Since $K:=\Gamma([0,1])$ is compact,
the assumption that $f$ is proper implies
$f^{-1}(K)$
is compact.
We let $\{t_k\}_{k=1}^\infty$ be a sequence on
$[0,1)$ converging to $1$.
Since $f^{-1}(K)$ is compact, 
the sequence 
$\{\sigma(t_k)\}_{k=1}^\infty$ 
has a convergent subsequence. 
\end{proof}

\begin{Example}\label{ex:T-comp}
We consider the image $f(\R)$ of the real analytic immersion 
\begin{equation}\label{eq:T}
 f:\R\ni t\mapsto (t, \sqrt{2}t)\in T^2
  :=\R^2/\Z^2.
\end{equation}
Although $f$ is not a proper map,
as will be seen below,
$f$ is $C^0$-arc-proper.
Let $\Gamma:[0,1]\to T^2$ be a 
continuous map satisfying $\Gamma([0,1))\subset f(\R)$. 
We denote by
 $\pi:\R^2\to T^2$ the canonical covering projection.
Then by the homotopy lifting property on covering spaces,
there exists a continuous map
$
\tilde \Gamma:I\to \R^2
$
such that $\pi\circ \tilde \Gamma=\Gamma$ on $I$
and 
$$
\tilde \Gamma(0)\in A,\qquad
A:=\{(x,y)\in \R^2\,;\, y=\sqrt{2}x\}.
$$
Since $A$ is a closed subset of $\R^2$,
we have $\tilde\Gamma(1)\in A$
and so
$$
\Gamma(1)=\pi\circ \tilde\Gamma(1)\in \pi(A)=f(\R).
$$
Thus, $f$ is $C^0$-arc-proper. 
\end{Example}

\begin{Exa}\label{ex:Kok}
Consider the analytic map
$
f(t):=(t,e^{-1/t})\,\, (t>0)
$
in the $xy$-plane.
As will be shown below, $f$ is a $C^\omega$-arc-proper immersion.
We set $S:=f((0,\infty))$.
Let $\sigma:[0,1) \rightarrow (0,\infty)$  be
a  $(C^\omega,f)$-extendable continuous map. 
Then there exists a real analytic map  $\Gamma:[0,1]\to \R^2$
such that $\Gamma(s)=f\circ \sigma(s)$ holds for $s\in [0,1)$.
We can write
$
\Gamma(s)=(x(s),y(s)) \,\, (s\in [0,1]).
$
Suppose that $\Gamma(1)\not \in S$.
Since $\overline{S}\setminus S=\{(0,0)\}$,
we have
$
\Gamma(1)=(0,0).
$
Moreover, since $\Gamma([0,1))\subset S$,
we have
$$
y(s)=e^{-1/x(s)},\quad x(s)>0 \qquad (s\in [0,1)).
$$
Then one can inductively check that 
$$
\lim_{s\to 1-0}\frac{d^k y}{ds^k}(s)=0\qquad (k=0,1,2,\dots)
$$
holds. 
This contradicts the fact that $y(s)$ is
a real analytic function at $s=1$.
Thus, we have $\Gamma(1)\in S$.
Then $x(1)>0$ and
$$
(x(s),y(s))=\Gamma(s)=f\circ \sigma(s)=(\sigma(s),e^{-1/\sigma(s)})
$$
hold. In particular,	
$\sigma(s)=x(s)$ converges to $x(1)$ as $s$ tends to $1$.
Thus, $f$ is a $C^\omega$-arc proper map.
By the reflection with respect to
the $y$-axis, the map $f$ can be
extended as a $C^\infty$-map defined by
$$
F(t):=
\begin{cases}
f(t)
&\mbox{($t>0$)},\\
(0,0) \qquad & \mbox{($t=0$)},\\
(t,e^{1/t})&\mbox{($t<0$)}.
\end{cases}  
$$
Then the map $F(t)$ itself is a smooth map on
$\R^2$, and $\sigma(t):=t$ ($t>0$) satisfies
$f\circ \sigma(t)=F(t)$ for $t>0$.
Since $\dy\lim_{t\to 0}\sigma(t)=0$ holds and
the origin $0$ does not
belong to $ (0,\infty) $,
$f$ is not $C^\infty$-arc proper.
So  $f$ gives an example which is
not $C^\infty$-arc-proper 
but is $C^\omega$-arc-proper.
\end{Exa}

\begin{Thm}\label{thm:first}
Let $f:X^m\to N^n$ be a $\op{DC}$-immersion which is
$C^\omega$-arc-proper as a continuous map. 
Then $f(X^m)$ is analytically complete.
In particular, if $f$ is a proper map,  
then $f(X^m)$ is analytically complete.
\end{Thm}

We will apply this to discuss the analytic extension of
G-catenoids in Section \ref{app:Gcatenoids}.

\begin{proof}
Let $\Gamma : [0, 1] \to N^{n}$ be an analytic map such
that $\Gamma ([0, \varepsilon ))\subset f(X^{m})$ 
for some $\varepsilon \in (0, 1]$. It is sufficient to show that 
$\Gamma (1) \in f(X^{m})$.
Given $0\leq a< b\leq 1$, a continuous map 
$\sigma : [a, b) \to X^{m}$ is called a {\it lift} of $\Gamma$ 
on the  interval [a, b) if 
$f\circ \sigma (t) = \Gamma (t)$ for all $t\in [a, b)$. 
By Lemma \ref{prop:Y}, if two lifts ${\sigma}_{1}$ 
and $\sigma_{2}$ of $\Gamma$ on $[a, b)$ satisfy 
$\sigma_{1} (a) = \sigma_{2} (a)$, 
then $\sigma_{1} (t)= \sigma_{2} (t)$ for all $t\in [a, b)$. 
By Lemma \ref{prop:X}, there exists $s_{0}$, $c_{0}$ with 
$0< s_{0}< c_{0}< \varepsilon $ and a lift 
$\sigma_{0}: [s_{0}, c_{0}) \to X^{m}$ of $\Gamma$ on $[s_{0}, c_{0})$.
Let 
$$
\mc C :=\{ c_{1}\in (s_{0}, 1] \, ; \, \text{there 
exists a lift $\sigma$ 
of $\Gamma$ on $[s_{0}, c_{1})$ with 
$\sigma (s_{0}) = \sigma_{0} (t_{0})$}\}. 
$$
Then $\mc C$ is nonempty, since $c_{0} \in \mc C$. 
Let $c$ be the supremum 
of $\mc C$. By definition, 
for each positive integer $k$,
there exists a lift $\sigma_k:I_k\to X^m$ 
of $\Gamma$ defined on an interval $I_k$ containing $(0, s_0-1/k)$
such that 
$\sigma_{k} (s_{0}) = \sigma_{0} (s_{0})$.
By Lemma \ref{prop:Y},
$\sigma_{k}=\sigma_{k+1}$ holds on $(0, s_0-1/k)$.
So there exists a unique continuous map $\gamma:(0, s_1)\to X$ 
satisfying $\gamma=\sigma_k$ on $(0, s_1-1/k)$.
If $s_1< 1$, then by  Lemma \ref{prop:Z}, $\gamma$ 
can be continuously extended on $[s_1, c+{\varepsilon}_{1} )$ 
for some ${\varepsilon}_{1} >0$ such that $f\circ \gamma = \Gamma$,  
which contradicts the maximality of $c$. 
So $s_1=1$. Taking a sequence $\{t_k\}_{k=1}^\infty$
on $[s_0,1)$ converging to $1$,
we have 
$$
\Gamma(1)=\lim_{k\to \infty}\Gamma(t_k)
=f(\lim_{k\to \infty}\gamma(t_k))\in f(X^m),
$$
proving the assertion.
\end{proof}

As a consequence,
we get the following:   

\begin{Corollary}\label{cor:deep}
{\it 
Let $X^m$ be a connected $C^\omega$-manifold
and  $f:X^m\to N^n$ a $C^\omega$-immersion.
If $f$ is a $C^\omega$-arc-proper map, 
then $f(X^m)$ is analytically complete.}
\end{Corollary}

Moreover, the following assertion holds:

\begin{Corollary}\label{cor:deep2}
{\it 
Let $X^m$ be a connected real analytic $\op{DC}$-manifold of dimension $m$
and  $f:X^m\to N^n$ a real analytic $\op{DC}$-immersion.
Suppose that $N^n$ is  properly $C^\omega$-embedded in
a $C^\omega$-manifold $L^l$ $(l\ge n)$, and
denote by $\iota:N^n\hookrightarrow L^l$ the inclusion map.
Then $f(X^m)$ is analytically complete in $N^n$
if and only if it is analytically complete in $L^l$.}
\end{Corollary}

\begin{proof}
The \lq\lq if''  part is obvious.
We suppose that $f(X^m)$ is analytically complete in $N^n$.
We consider a real analytic map
$
\Gamma:[0,1]\to L^l
$
such that $\Gamma ([0,\epsilon) )\subset f(X^m)$
for sufficiently small $\epsilon>0$.
Then 
$
\Gamma ([0,\epsilon)) \subset f(X^m)\subset N^n
$
holds. Since $N^n$ is analytically complete in $L^l$,
we have $\Gamma([0,1])\subset N^n$.
Since $f(X^m)$ is an analytically complete subset of $N^n$,
we also have $\Gamma([0,1])\subset f(X^m)$.
Thus $f(X^m)$ is analytically complete in $L^l$.
\end{proof}

\section{Analytic extensions of catenoids in $S^3_1$}
\label{app:Gcatenoids}

In this section, we define
\lq\lq real analytic space-like CMC-1 $\op{DC}$-immersions\rq\rq,
and we show that each G-catenoid or its appropriate analytic 
extension in the de Sitter 3-space $S^3_1$ 
belongs to this class and
that all of the images of
these maps are analytically complete.

\subsection{Space-like CMC-1 $\op{DC}$-immersions}

We give the following definition.

\begin{Definition}\label{def:CMC-1}
Let $f:X^2\to S^3_1$ be a 
real analytic $\op{DC}$-immersion 
defined on a connected $2$-dimensional real analytic
$\op{DC}$-manifold $X^2$.
Then $f$ is called 
a {\it real analytic space-like CMC-1 $\op{DC}$-immersion}
if there exists an open dense subset $O$ of $X^2\setminus \Sigma$
such that the restriction of $f$
to $O$ is a 
space-like constant mean curvature one (i.e. CMC-1) 
immersion, where $\Sigma$ is the set of $\op{DC}$-points in $X^2$.
\end{Definition}

Similarly, real analytic  ZMC $\op{DC}$-immersion
in $\R^3_1$ can be defined (see Appendix A for details).
We prove the following:

\begin{Proposition}\label{prop:AB-1}
Let $f:X^2\to S^3_1$ 
be a real analytic $\op{DC}$-immersion
defined on a connected $2$-dimensional real analytic
$\op{DC}$-manifold $X^2$.
Suppose that there exists a non-empty open subset
$O$ of $X^2$ such that the restriction
$f|_O:O\to S^3_1$ gives a real analytic
space-like CMC-1 $\op{DC}$-immersion.
Then $f$ is a
real analytic space-like CMC-1 $\op{DC}$-immersion. 
\end{Proposition}

\begin{proof}
We let $\Sigma$ be the set of $\op{DC}$-points in $X^2$.
We fix a point $p\in X^2\setminus \Sigma$
arbitrarily.
Since $X^2$ is connected,
there exists a continuous map
$\gamma:[0,1]\to X^2$ such that
$\gamma(0)\in O$ and $\gamma(1)=p$.
Since $\gamma([0,1])$ is compact,
we can take a partition
$
0=t_0<t_1<\cdots<t_k=1
$
and a family of inverse $\op{DC}$-coordinate systems $\{(U_i,\Phi_i)\}_{i=0}^k$
such that 
\begin{itemize}
\item $U_i$ ($i=0,\ldots,k$) are connected
and $\gamma(t_i)\in \Phi_i(U_i)$,
\item $\Phi_0(U_0)\subset O$, and
\item $\Phi_{i-1}(U_{i-1})\cap \Phi_i(U_i)$ 
is non-empty for $i=1,\dots,k$.
\end{itemize}
We know that $f|_{\Phi_0(U_0)}$
is 
a 
space-like CMC-1 DC-immersion 
on $\Phi_{0}(U_{0})$.
If the assertion fails, then there exists a number $j(\ge 1)$
so that $f|_{\Phi_{j-1}(U_{j-1})}$ is 
a space-like CMC-1 DC-immersion, but 
$f|_{\Phi_{j}(U_{j})}$ is not a space-like CMC-1 DC-immersion.
Since $U_{j-1}$ and $U_{j}$ are open submanifolds
of  $S^{1}\times \R$, we may assume that
they have local coordinate systems. 
Let $A_{i}$ and  $B_i$ ($i=0,\ldots n$)
be the two particular analytic functions  
defined on $U_i$ in \cite[(2.3)]{HKKUY}
(in \cite{HKKUY}, $\alpha$, $\beta$ correspond 
to $A_i$ and $B_i$ in our paper),
which have the property that the mean curvature $H$ of 
$f_i:=f\circ \Phi_i$ satisfies 
\begin{equation}\label{eq:H}
H=\pm A_i/(2|B_i|^{3/2}) \qquad (i=0,\ldots n)
\end{equation}
on an open dense subset of $U_i$, and
$B_i$ is positive (resp. zero, negative) at a
space-like (resp. light-like, time-like) point.
We take a non-empty open subset $W$ of $X^2$
such that $W\subset \Phi_{j-1}(U_{j-1})\cap \Phi_{j}(U_{j})$.
Since $\Sigma$ is a discrete subset of $X^2$,
we may assume that $W$ does not contain any $\op{DC}$-points.
Since $f$ is a real analytic space-like DC-immersion,
by taking $W$ sufficiently small, we may assume that
$f_{j-1}$ gives a space-like CMC-1  immersion on $\Phi^{-1}_{j-1}(W)$.
In particular,
$B_{j-1}$ is positive-valued on $\Phi_{j-1}^{-1}(W)$.
So, \eqref{eq:H} implies
\begin{equation}\label{eq:H2}
A_{j-1}(x)^2=4B_{j-1}(x)^3\qquad (x\in \Phi_{j-1}^{-1}(W)).
\end{equation}
Since the property that \lq\lq space-like with $|H|=1$\rq\rq\
does not depend on the choice of local coordinate system,
we have
\begin{equation}\label{eq:H3}
A_{j}(x)^2=4B_{j}(x)^3\qquad (x\in \Phi_{j}^{-1}(W)).
\end{equation}
Since $\Phi_{j}^{-1}(W)\subset U_j$
and $U_j$ is connected,
the real analyticity of $f_j$ 
implies that $A_{j}^2=4B_{j}^3$ holds on $U_j$.
As a consequence, $B_{j}\ge 0$ holds on $U_{j}$
and
$f$ gives a space-like CMC-1 immersion 
on an open dense subset of $\Phi_j(U_j)$, a contradiction.
\end{proof}

\begin{Remark}\label{rmk:add}
As a consequence of \cite[Theorem D]{UY3},  
if a real analytic CMC-1
immersion $f:(U;u,v)\to S^3_1$ 
admits a light-like point $p\in U$, then
there exists a regular map $\sigma:(-\epsilon,\epsilon)\to U$
($\epsilon>0$) such that 
\begin{itemize}
\item $\sigma(0)=p$,
\item $\sigma(t)$ ($|t|<\epsilon$) are light-like points of $f$, and
\item  $f\circ \sigma$ is a light-like geodesic segment in $S^3_1$.
\end{itemize}
This is the reason why the following analytic extensions of
CMC-1 catenoids in $S^3_1$ often contain light-like geodesic segments.
\end{Remark}

\subsection{The classification of G-catenoids}

A \emph{geometric catenoid} (or a \emph{G-catenoid} for short)
is a space-like CMC-1 face in de Sitter $3$-space $S^3_1$ 
(cf. \cite{F} and \cite{FRUYY})
which is invariant under the action of a one-parameter subgroup $G$ of isometries 
of $S^3_1$ with a common fixed point.
Such a subgroup $G$  fixes one or two geodesics, which are called the \emph{axes}
of the subgroup $G$.
We say that a G-catenoid is of \emph{type T} (resp.\ \emph{L}, \emph{S})
if the corresponding axes are time-like (resp.\ light-like, space-like) geodesics.
The analytic part of a G-catenoid can be considered as a CMC-1 face
in the sense of \cite{FKKRUY2}, that is, it can be expressed as 
\[
    f = (x_0,x_1,x_2,x_3),\qquad \begin{pmatrix}
				  x_0+x_3 & x_1+\imag x_2 \\
				  x_1-\imag x_2 & x_0-x_3
				 \end{pmatrix} = F\begin{pmatrix}
						   1 & \hphantom{-}0 \\
						   0 & -1
						  \end{pmatrix}
            					  F^*,
\]
where $F$ is a holomorphic immersion into $\SL(2,\C)$
defined on a certain Riemann surface.
Such an $F$ is called the \emph{holomorphic lift} of the surface
(cf.\ \cite{FKKRUY2} for details).
G-catenoids were completely classified 
by the last author \cite{Y}
as in Table~\ref{tab:G-catenoids}.
We remark that a surface of type LE (resp.\ LH)
 with holomorphic lift $F_L(\alpha,z)$ is congruent
 to $F_L(1,z)$ (resp.\ $F_L(-1,z)$).

\begin{table}[h]
 \begin{center}
  \begin{tabular}{|c||c|c|c|}
   \hline
   axis   & E-type & P-type & H-type\\ \hline \hline
   \begin{minipage}[t]{1.6cm}
    time-like \\
    (type T)
   \end{minipage}
    &
   (TE)
   \begin{minipage}[t]{2.3cm}
    $F_T\left(\frac{1-\mu^2\mathstrut}{2},z\right)$\newline
    \hfill $\bigl(\mu\in\R_+\setminus\{1\}\bigr)$
   \end{minipage}
    & (TP) $F_T(0,z)$ 
    & (TH) 
   \begin{minipage}[t]{2.3cm}
    $F_T\left(\frac{1-\mu^2\mathstrut}{2},z\right)$\newline
    \hfill $\bigl(\mu\in\imag\R_+\bigr)$
   \end{minipage}\\
   \hline
   \begin{minipage}[t]{1.6cm}
    light-like \\
    (type L)
   \end{minipage}&
    (LE) 
   \begin{minipage}[t]{2.3cm}
    $F_L\bigl(\alpha, z\bigr)$\newline
    \hfill $\bigl(\alpha>0\bigr)$
   \end{minipage}& ---
    & (LH) 
   \begin{minipage}[t]{2.3cm}
    $F_L\bigl(\alpha, z\bigr)$\newline
    \hfill $\bigl(\alpha<0\bigr)$
   \end{minipage}
   \\ \hline
   \begin{minipage}[t]{1.6cm}
    space-like \\
    (type S)
   \end{minipage}
    &
   (SE)
   \begin{minipage}[t]{2.3cm}
    $F_S\left(\frac{1-\mu^2\mathstrut}{2},z\right)$\newline
    \hfill $\bigl(\mu\in\R_+\setminus \{1\}\bigr)$
   \end{minipage}
    & (SP) $F_S(0,z)$ 
    & (SH) 
   \begin{minipage}[t]{2.3cm}
    $F_S\left(\frac{1-\mu^2\mathstrut}{2},z\right)$\newline
    \hfill $\bigl(\mu\in\imag\R_+\bigr)$
   \end{minipage}\\
   \hline
  \end{tabular}
 \end{center}
 \caption{The classification of G-catenoids in \cite{Y}.
 $F_T$, $F_L$ and $F_S$ denote the holomorphic lifts
 of  surfaces, which are defined in 
 \cite[Lemma 4.2, Section 4]{Y}.}\label{tab:G-catenoids}
\end{table}

Since the monodromy matrix of the secondary Gauss map of a G-catenoid
takes values in $\op{SU}(1,1)$
whose trace is real and less than (resp. greater than, equal to) two,
we call such a matrix elliptic (resp. hyperbolic, parabolic).
So we call $f^{\mu}_{\mathrm{TE}}$ a G-catenoid of type {\rm TE},
since the monodromy of its secondary  Gauss map  is elliptic.

From now on, we show that each G-catenoid 
or
its suitable analytic extension is analytically complete in $S^3_1$.
Since $S^3_1$ is a properly embedded real analytic submanifold
of $\R^4$, all those surfaces are also analytically complete in $\R^4$
(cf. Corollary \ref{cor:deep2}).

\subsection{G-catenoids of type T}

As mentioned above,
there are three subclasses of G-catenoids of type T.
In this subsection, we show that the images of all of 
them are analytically complete. 

\subsubsection{G-catenoids of type TE}
We set $T^1:=\R/2\pi \Z$.
For each 
$\mu\in \R_+ \setminus \{1\}$,
we consider a real analytic map
$$
   f_{\op{TE}}^\mu 
\colon{} \R\times T^1 \ni
(s,\theta) \mapsto (x_0(s),x_1(s,\theta),x_2(s,\theta),x_3(s))\in 
S^3_1,
$$
which gives a G-catenoid of type {\rm TE}
, 
where
\begin{align*}
f^{\mu}_{\mathrm{TE}}(s, \theta) =
 \begin{pmatrix}
  1 & 0 & 0 & 0 \\ 0 & \cos \theta & -\sin \theta & 0 \\ 
 0 & \sin \theta & \cos \theta & 0 \\ 0 & 0 & 0 & 1
 \end{pmatrix}
\Gamma^{\mu}_{\mathrm{TE}}(s),\quad
\Gamma^{\mu}_{\mathrm{TE}}(s):=\begin{pmatrix}
 x_0(s) \\ 
\frac{1- \mu^2}{2 \mu} \sinh \mu s \\ 
0 \\ 
x_3(s) 
\end{pmatrix}, 
\end{align*}
and
\begin{equation}\label{eq:x0x3-TE} 
\begin{aligned}
x_0(s)&=\sinh s \cosh \mu s-\frac{\left(\mu^2+1\right) \cosh s \sinh \mu s}{2 \mu},\\
x_3(s)&=\cosh s\cosh \mu s-\frac{\left(\mu^2+1\right) \sinh s \sinh \mu s}{2 \mu}.
 \end{aligned}
\end{equation}
Here $\Gamma_{\mathrm{TE}}^\mu(s)$ ($s\in \R$) is a
profile curve of $f^{\mu}_{\mathrm{TE}}$
lying on $S^2_1 = S^3_1 \cap \{ x_2=0 \}$. 
The axes of $f_{\op{TE}}^\mu$ 
are $\{(\sinh t,0,0,\pm\cosh t)\;;\;t\in\R\}$.
The singular point set of $f_{\op{TE}}^\mu$ 
is  $\{(0,\theta)\in \R\times T^1\}$,
whose image consists of one point  $\{(0,0,0,1)\}$
(see Figure \ref{fig:typeTE}).

\begin{figure}[hbt]%
 \begin{center}
       \includegraphics[width=3.5cm]{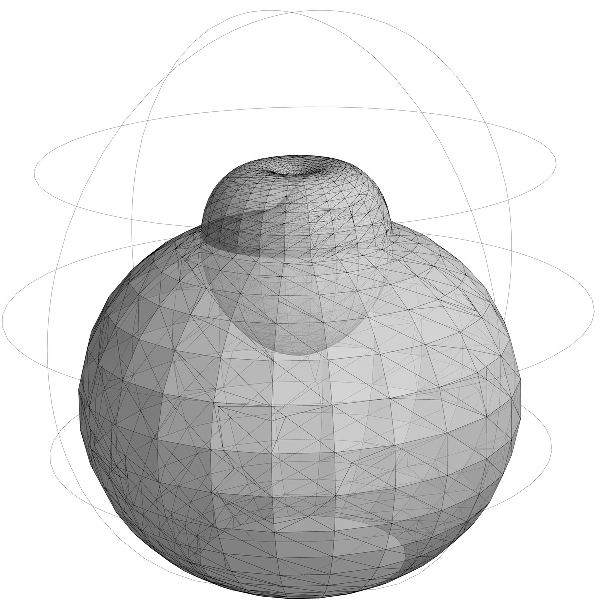} \hskip 2cm
       \includegraphics[width=3.5cm]{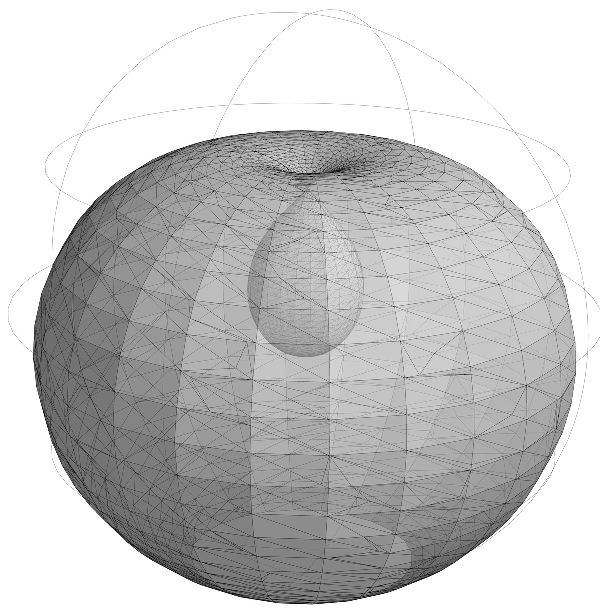}  
\end{center}
  \caption{G-catenoids of type TE with $0<\mu<1$ (left) and $\mu>1$ (right)
in the stereographic hollow ball model (cf. \cite[(1)]{FKKRUY3}
)
}%
\label{fig:typeTE}
\end{figure}

The following assertion can be checked easily:

\begin{Lemma}\label{lem:5-1}
For each 
$\mu\in \R_+ \setminus \{1\}$,
the plane curve $s\mapsto (x_0(s),x_3(s))$ is a regular curve on $\R$.
\end{Lemma}

\begin{proof}
In fact, we have
$$
\pmt{x_0'(s) \\ x'_3(s)}=
\frac{\mu^2-1}{2\mu}
\pmt{ \cosh s & \sinh s \\
 \sinh s &\cosh s 
}
\pmt{
 -\mu \cosh \mu s \\
 \sinh \mu s
},
$$
which implies the assertion.
\end{proof}

\begin{Proposition}
The map
$f_{\op{TE}}^\mu$ is a 
real analytic space-like CMC-1 $\op{DC}$-immersion,
and its image is analytically complete.
\end{Proposition}

Note that $f_{\op{TE}}^\mu$ is a W-catenoid with 
elliptic monodromy as in \cite{FKKRUY2}.

\begin{proof}
For each $\mu\in \R_+\setminus \{1\}$,
we set
$$
\mc X_{\op{TE}}^\mu:=
\left
\{(\xi,\eta,s)\in \R^3\,;\, 
\xi^2+\eta^2=\left(\frac{1-\mu^2}{2\mu}\right)^2\sinh^2
\mu s
\right\}.
$$
It can be easily checked that $\mc X_{\op{TE}}^\mu\setminus \{\mb 0\}$
is a real analytic submanifold of $\R^3$, and so
it is an admissible subset of $\R^3$.
Since the map
$$
\R^3\ni(x,y,z)\mapsto \left(x,y,\frac1\mu\sinh^{-1}\left(\frac{2\mu z}{1-\mu^2}\right)\right)\in \R^3
$$
gives a surjective extended parametrization of
$\mc X_{\op{TE}}^\mu$, 
by Proposition \ref{prop:Xstr},
$\mc X_{\op{TE}}^\mu$
has the structure of 
a $\op{DC}$-submanifold of $\R^3$ 
so that $\mb 0\in \mc X_{\op{TE}}^\mu$ 
is a $\op{DC}$-point. 
Moreover,
the inclusion map
 $\iota:\mc X_{\op{TE}}^\mu\hookrightarrow \R^3$
is a $\op{DC}$-immersion.
By Lemma \ref{lem:5-1},
the map
$
\Psi:\R^3 \ni (\xi,\eta,s)
\mapsto (x_0(s),\xi,\eta,x_3(s))\in S^3_1
$
is an immersion, and so
the composition
$\tilde f_{\op{TE}}:=\Psi\circ \iota$
is also a $\op{DC}$-immersion
(cf. Corollary \ref{cor:Xm}).
Since
$$
f_{\op{TE}}^\mu(s,\theta)=\tilde f_{\op{TE}}^\mu(x_1(s,\theta),x_2(s,\theta),s),
$$
the image of $\tilde f_{\op{TE}}^\mu$ coincides with that of $f_{\op{TE}}^\mu$.
The map $\tilde f_{\op{TE}}^\mu$ is 
a 
real analytic space-like CMC-1 $\op{DC}$-immersion.
It can be easily checked that $\tilde f_{\op{TE}}^\mu$ is a proper map.
So, Theorem~\ref{thm:first}
yields that
the image of $\tilde f_{\op{TE}}^\mu$ is analytically complete.
\end{proof}

\subsubsection{G-catenoids of type TP}
A G-catenoid of type {\rm TP}  
can be expressed as
 \begin{align*}
   f_{\op{TP}}
 \colon{} \R\times T^1\ni
(s,\theta) \mapsto (x_0(s),x_1(s,\theta),x_2(s,\theta),x_3(s))\in S^3_1,
 \end{align*}
where
\begin{align*}
f_{\op{TP}}(s, \theta)=
 \begin{pmatrix}
  1 & 0 & 0 & 0 \\ 0 & \cos \theta & -\sin \theta & 0 \\ 
 0 & \sin \theta & \cos \theta & 0 \\ 0 & 0 & 0 & 1
 \end{pmatrix}
\Gamma_{\op{TP}}(s), \quad 
\Gamma_{\op{TP}}(s) :=
\begin{pmatrix}
x_0(s)  \\ 
s/2 \\ 
0 \\ 
x_3(s)  
\end{pmatrix}, 
\end{align*}
and
\begin{equation}\label{eq:x0x3-TP}
 x_0(s)=\sinh s - \frac{s}{2} \cosh s,\qquad
x_3(s)=\cosh s - \frac{s}{2} \sinh s.  
\end{equation}
Since the monodromy matrix of the secondary 
Gauss map of $f_{\op{TP}}$ is parabolic,
we call $f_{\op{TP}}$ a 
G-catenoid of type {\rm TP}.
Here
$\Gamma_{\op{TP}}(s)$ ($s\in \R$) is the profile curve
of $f_{\op{TP}}$ lying on $S^2_1 = S^3_1 \cap \{ x_2=0 \}$.
The axes of $f_{\op{TP}}$ are the same as the axes of $f_{\op{TE}}^\mu$.
The singular point set of  $f_{\op{TP}}$ is  
$	    \{(0,\theta)\in \R\times T^1\}$,
whose image consists of one point $\{(0,0,0,1)\}$.
Note that $f_{\op{TP}}$ is the limit of $f_{\op{TE}}^\mu$ as $\mu \to 0$. 
We set
$$
\mc X_{\op{TP}}:=
\left\{
(\xi,\eta,s)\in \R^3\,;\, \xi^2+\eta^2=\frac{s^2}{4}
\right\}.
$$
Like as in the case of $\mc X_{\op{TE}}$,
the subset $\mc X_{\op{TP}}$ has the structure of a $\op{DC}$-submanifold
of $\R^3$ such that $(0,0,0)$ is a $\op{DC}$-point.
It can be easily checked that 
$\R\ni s\mapsto (x_0(s),x_3(s))\in \R^2$ is a 
regular curve.
So 
$$
\Psi:\R^3\ni (\xi,\eta,s)
\mapsto (x_0(s),\xi,\eta,x_3(s))\in S^3_1
$$
is an immersion.
If we denote the inclusion map by
$\iota:\mc X_{\op{TP}}\hookrightarrow \R^3$,
then the composition
$\tilde f_{\op{TP}}:=\Psi\circ \iota$
is a $\op{DC}$-immersion.
Since
$
f_{\op{TP}}(s,\theta)=\tilde f_{\op{TP}}(x_1(s,\theta),x_2(s,\theta),s),
$
the image of $\tilde f_{\op{TP}}$ coincides with that 
of $f_{\op{TP}}$. 
On the other hand, it is easily checked that
$\tilde f_{\op{TP}}$ is a proper map.
Applying
Theorem \ref{thm:first}
yields the following:

\begin{Proposition}
The map $f_{\op{TP}}$ is a 
real analytic space-like CMC-1 $\op{DC}$-immersion, and
its image is analytically complete.
\end{Proposition}

Note that $f_{\op{TP}}$ is
a W-catenoid with parabolic monodromy as in \cite{FKKRUY2}.

\subsubsection{G-catenoids of type TH}
For each 
$\nu\in \R_+ \setminus \{1\}$,
we consider the following real analytic map
\begin{equation*}
   f_{\op{TH}}^\nu  \colon{} \R\times T^1
\ni (s,\theta) \mapsto (x_0(s),x_1(s, \theta),x_2(s, \theta),x_3(s))
 \in S^3_1
\end{equation*}
which gives a G-catenoid of type TH, where
\begin{align*}
f^{\nu}_{\op{TH}}(s, \theta):=
 \begin{pmatrix}
  1 & 0 & 0 & 0 \\ 0 & \cos \theta & -\sin \theta & 0 \\ 
 0 & \sin \theta & \cos \theta & 0 \\ 0 & 0 & 0 & 1
 \end{pmatrix}
\Gamma^{\nu}_{\op{TH}}(s), \quad 
\Gamma^{\nu}_{\op{TH}}(s) :=
\begin{pmatrix}
x_0(s)  \\ 
\frac{\nu^2+1}{2 \nu} \cos(\nu s) \\ 
0 \\ 
x_3(s)  
\end{pmatrix},
\end{align*}
and
\begin{equation}\label{eq:x0x3-TH}
\begin{aligned}
x_0 (s) &:=\sin (\nu s) \sinh s
-\frac{\left(\nu^2-1\right) \cos(\nu s) \cosh s}{2 \nu}, \\
x_3 (s) &:=\sin (\nu s) \cosh s
-\frac{\left(\nu^2-1\right) \cos(\nu s) \sinh s}{2 \nu}.
 \end{aligned}
\end{equation}
Since the monodromy matrix of the secondary Gauss map of 
$f^{\mu}_{\mathrm{TH}}$ is hyperbolic,
we call $f^{\mu}_{\mathrm{TH}}$ a G-catenoid of type {\rm TH}.
Here $\Gamma^\nu_{\op{TH}}(s)$ ($s\in \R$) is a profile curve
of $f^\nu_{\op{TH}}$
lying on $S^2_1 = S^3_1 \cap \{ x_2=0 \}$.
The axes of $f_{\op{TH}}^\nu$ are the same as the axes of $f_{\op{TE}}^\mu$.

The singular point set of  $f_{\op{TH}}^\nu$
is given by
       \[
	    \Sigma:=\bigcup_{k\in\Z}\Sigma_k,\qquad
            \Sigma_k := \left\{
                       \left(\left(\frac{1}{2}+k\right)\frac{\pi}{\nu},\theta\right)\,;
                         \theta\in T^1
                       \right\},
       \]
       and the image of each connected component $\Sigma_k$ consists
       of one point
       \[
     \left(
 (-1)^{k+1}\sinh\left(\left(\frac{1}{2}+k\right)\pi\right),
0, 0,
	          (-1)^{k+1}\cosh\left(\left(\frac{1}{2}+k\right)
\pi\right)\right).
       \]
Unlike the case of $f_{\op{TE}}^\mu$ and $f_{\op{TP}}$,
the map $f_{\op{TH}}^\nu\colon{}\R\times T^1\to S^3_1$ is not
a proper map. In fact, the image of the map accumulates to the set
       \[
	  \mc L = \mc L_+\cup \mc L_-, \qquad 
          \mc L_\pm := \{(x_0,x_1,x_2,x_3)\in S^3_1\,;\, x_0\mp x_3=0\}.
       \]
Figures for the image of $f_{\op{TH}}^\nu$ are given
in \cite[Fig. 3 (page 35)]{FKKRUY2}.
Like as in 
Example 
\ref{ex:H0},
the set
$$
\mc X_{\op{TH}}^\nu:=
\left\{
(\xi,\eta,s)\in \R^3\,;\, \xi^2+\eta^2=\frac{(1+\nu^2)^2\cos^2(\nu s)}{4\nu^2}
\right\}
$$
has the structure of a $\op{DC}$-submanifold 
of $\R^3$ whose $\op{DC}$-points are
 $(0,0,(\pi/2+\pi k)/\nu)$ 
($k\in \Z$).
Using the functions \eqref{eq:x0x3-TH}, 
we define  
$$
\Psi:\R^3 \ni (\xi,\eta,s)
\mapsto (x_0(s),\xi,\eta,x_3(s))\in S^3_1,
$$
which is a real analytic diffeomorphism into $S^3_1$ by the following lemma:

\begin{Lemma}\label{lem:5-2}
For each $\nu\in\R_+\setminus\{1\}$, 
the plane curve $s\mapsto (x_0(s),x_3(s))$ is a regular curve on $\R$.
\end{Lemma}

\begin{proof}
In fact, the assertion follows from the following identity
$$
\pmt{x_0'(s) \\ x'_3(s)}=
\frac{\nu^2+1}{2\nu}
\pmt{ \sinh s & \cosh s \\
 \cosh s  & \sinh s 
}
\pmt{
 \cos(\nu s) \\
 \nu\sin(\nu s)
}.
$$
\end{proof}
By this lemma,
$\tilde f_{\op{TH}}^\nu:=\Psi|_{\mc X_{\op{TH}}^\nu}$
is a $\op{DC}$-immersion. 
Since
$$
f_{\op{TH}}^\nu(s,\theta)
=\tilde f_{\op{TH}}^\nu(x_1(s,\theta),x_2(s,\theta),s),
$$
the image of $\tilde f_{\op{TH}}^\nu$ coincides 
with that of $f_{\op{TH}}^\nu$.
So $f_{\op{TH}}^\nu$ is a 
space-like CMC-1
immersion on an open dense subset in $\mc X_{\op{TH}}^\nu$.
So we can prove the following:

\begin{Proposition}
The map
$f_{\op{TH}}^\nu$ is a 
real analytic space-like CMC-1 $\op{DC}$-immersion
and its image is analytically complete.
\end{Proposition}

\begin{proof}
Although $\tilde f_{\op{TH}}^\nu$ is not a proper map,
$\tilde f_{\op{TH}}^\nu$ is a $C^0$-arc-proper map:
In fact, let $\hat\gamma:[0,1]\to S^3_1$ be a continuous
map, and suppose that
$
\gamma(t)=(\xi(t),\eta(t),s(t))
$
($t\in [0,1)$) satisfies 
$\hat\gamma(t)=\tilde f_{\op{TH}}^\nu\circ \gamma(t)$ for each
$t\in [0,1)$.
Since the second and the third components of
$\hat\gamma(t)$ are $\xi(t),\eta(t)$,
the limits $\dy\lim_{t\to 1-0}\xi(t)$ and $\dy\lim_{t\to 1-0}\eta(t)$
exist. Since
$$
\xi(t)^2+\eta(t)^2=\frac{(1+\nu^2)^2\cos^2 s(t)}{4\nu^2}
\qquad (t\in [0,1)),
$$
$\dy\lim_{t\to 1-0}\cos s(t)$ exists.
By Lemma~\ref{lem:C} in the second appendix, $\dy\lim_{t\to 1-0}s(t)$ also exists.
Thus, we  can show the existence of
the limit $\dy\lim_{t\to 1-0} \gamma(t)$, proving the
$C^0$-arc-properness of $\tilde f_{\op{TH}}^\nu$.
So the assertion follows from Theorem~\ref{thm:first}, 
because $\tilde f_{\op{TH}}^\nu$ is a $\op{DC}$-immersion.
\end{proof}

\begin{Remark}\label{rmk:infty}
If $\nu=1$, we have that
$$
x_0=\sin s \sinh s,\quad x_1=\cos \theta \cos s,
\quad x_2=\sin \theta \cos s,
\quad x_3= \sin s \cosh s.
$$
Then $f_{\op{TH}}^\nu(k\pi,\theta)=
(0,(-1)^k\cos \theta, (-1)^k\sin \theta,0)$ 
($\nu=1$) for each $k\in \Z$. 
In particular, $f_{\op{TH}}^\nu$ ($\nu=1$)
takes the same value infinitely many times.
\end{Remark}

From now on, we consider G-catenoids of type $S$ and $P$,
all of which have non-trivial analytic extensions:

\subsection{G-catenoids of type S}\par

\subsubsection{G-catenoids of type SE}
For each $\mu\in \R_+\setminus \{1\}$,
we consider a map
\begin{equation*}
    f_{\op{SE}}^\mu \colon \R^2
\ni (s,\theta) \mapsto (x_0(s,\theta),x_1(\theta),
x_2(\theta),x_3(s,\theta)) 
\in S^3_1
\end{equation*}
given by
\begin{align*}
f_{\op{SE}}^\mu (s,\theta)  
=
\begin{pmatrix}
 \cosh s & 0 & 0 & \sinh s \\ 
 0 & 1 & 0 & 0 \\ 
 0 & 0 & 1 & 0 \\ 
 \sinh s & 0 & 0 & \cosh s
\end{pmatrix}
\Gamma^\mu_{\op{SE}}(\theta), 
 \quad 
\Gamma^\mu_{\op{SE}}(\theta) :=
\begin{pmatrix}
 -\frac{\mu ^2-1}{2 \mu} \cos \mu \theta \\ 
x_1(\theta) \\
x_2(\theta) \\
0
\end{pmatrix},
\end{align*}
where
\begin{equation}\label{eq:f-SE}
 \begin{aligned}
x_1(\theta)&=-\frac{(\mu^2+1) \cos \theta  
\cos \mu\theta}{2 \mu}-\sin\theta \sin \mu\theta, \\
x_2(\theta)&=-\frac{(\mu^2+1) \sin \theta  
\cos \mu \theta}{2\mu}+\cos \theta  \sin \mu \theta.
\end{aligned}
\end{equation}
Then this gives
a G-catenoid of type {\rm SE}.
Here,  $\Gamma^{\mu}_{\op{SE}}(\theta)$ ($\theta \in \R$) 
is the profile curve of
$f^{\mu}_{\op{SE}}$
lying on $S^2_1 = S^3_1 \cap \{ x_3=0 \}$.
The axis of $f_{\op{SE}}^\mu$ is $\{(0,\cos t,\sin t,0)\;;\;t\in T^1\}$.
The singular point set of $f_{\op{SE}}^\mu$ is 
       \[
	  \Sigma:=\{(s,\theta)\,;\, \cos \mu \theta=0\}
                 = \bigcup_{k\in\Z} \Sigma_k,\qquad
           \Sigma_k:=\left\{
            \left(s,\frac{1}{\mu}
           \left(\frac{\pi}{2}+k\pi\right)\right)\,;\, s\in\R\right\},
       \]
       and the image of each $\Sigma_k$ consists of  one point
       \begin{equation}\label{eq:SE-sing}
	  \pt{P}_k:=f_{\op{SE}}^\mu(\Sigma_k) =
              \left(0,\gamma_{\op{SE}}^\mu \left(\frac{1}{\mu}
               \left(\frac{\pi}{2}+k\pi\right)\right),
             0\right),
       \end{equation}
where 
$ \gamma^{\mu}_{\op{SE}} (\theta) := (x_1(\theta), x_2(\theta)) \in \R^2$. 
The limit point set of  $f_{\op{SE}}^\mu$ is 
       \[
	  \mc L:= \bigcup_{k\in\Z} \mc L_k,\quad
          \mc L_k:=\left\{\left(
                 u,\gamma_{\op{SE}}^{\mu}\left(\frac{\pi}{2\mu}
              \left(2k+1\right)\right),\pm u\right)\in S^3_1\,;\,
                 u\in\R\right\}.
       \]
We prove that the image of this map $f^{\mu}_{\op{SE}}$ 
has an analytic extension:
For any $\mu \in \R_+ \setminus \{ 1\}$, 
define a subset  $\mc X_{\op{SE}}^\mu\subset\R^3$ 
by 
\begin{equation}\label{eq:X-SE} 
\mc X_{\op{SE}}^\mu:=\left\{(\xi,\eta,\theta)\in \R^3\,;\,\xi\eta=
\left(\frac{(\mu^2-1)\cos \mu\theta}{2\mu}\right)^2
\right\}.
\end{equation}
 By  \eqref{eq:f-SE}, we have
\begin{equation}\label{eq:cosSE}
   |\gamma^\mu_{\op{SE}}(\theta)|^2
= 1+\left(\frac{\mu^2-1}{2\mu}\right)^2\cos^2 \mu \theta.
\end{equation}
Thus, if we set 
$$\xi:=x-y,\quad
\eta:=x+y,\quad
z:=\left(\frac{\mu^2-1}{2\mu}\right) \cos \mu \theta,
$$
then $-x^2+y^2+z^2=0$ holds.
Using this expression, it can be easily checked that 
$\mc X_{\op{SE}}^\mu$ 
has the structure of a $\op{DC}$-submanifold
with a countably infinite number of $\op{DC}$-points, whose
inclusion map is a $\op{DC}$-immersion.
Using the functions \eqref{eq:f-SE}, we set 
$$
	  \tilde f_{\op{SE}}^\mu\colon{}\mc X_{\op{SE}}^\mu 
\ni (\xi,\eta,\theta)
          \mapsto \left(\frac{\xi+\eta}{2},x_1(\theta),x_2(\theta),
                        \frac{\xi-\eta}{2}\right)\in S^3_1.
$$
It holds that
$$
f^\mu_{\op{SE}}(s,\theta)=\tilde f_{\op{SE}}^\mu
\left(-\frac{(\mu ^2-1) e^s \cos \mu  \theta}{2 \mu},
-\frac{(\mu ^2-1) e^{-s} \cos \mu  \theta}{2 \mu },\theta\right).
$$
If we set
$$
\check f^\mu_{\op{SE}}(s,\theta)=\tilde f_{\op{SE}}^\mu
\left(\frac{(\mu ^2-1) e^s \cos \mu  \theta}{2 \mu },
\frac{(\mu ^2-1) e^{-s} \cos \mu  \theta}{2 \mu },\theta\right),
$$
then we have
$$
\check f_{\op{SE}}^\mu=
\pmt{
-1 & 0 & 0 & 0 \\
0 & 1 & 0 & 0 \\
0 & 0 & 1 & 0 \\
0 & 0 & 0 & -1 
}f^\mu_{\op{SE}},
$$
which implies the image of $\check f_{\op{SE}}^\mu$ is congruent to
that of $f_{\op{SE}}^\mu$ in $S^3_1$.
It can be easily checked that
\begin{equation}\label{eq:SE-E}
\tilde f_{\op{SE}}^\mu(\mc X_{\op{SE}}^\mu)=
f_{\op{SE}}^\mu(\R^2)\cup \mc  L\cup \check f_{\op{SE}}^\mu(\R^2).
\end{equation}
We prove the following: 

\begin{Lemma}\label{lem:5-3}
The plane curve $\theta\mapsto (x_1(\theta),x_2(\theta))$ 
is a regular curve on $\R$.
\end{Lemma}

\begin{proof}
In fact,  the assertion follows from the following identity
$$
\pmt{x_1'(\theta) \\ x'_2(\theta)}=
\frac{\mu^2-1}{2\mu}
\pmt{ \cos \theta & -\sin \theta \\
  \sin \theta  & \cos \theta
}
\pmt{
\mu \sin\mu \theta \\ \cos\mu \theta
}.
$$
\end{proof}

We obtain the following:

\begin{Proposition}\label{prop:SE-noext}
The map $\tilde f_{\op{SE}}^\mu$ 
is a 
real analytic space-like CMC-1 $\op{DC}$-immersion
whose image is analytically complete.
\end{Proposition}

\begin{proof}
By Lemma \ref{lem:5-3},
the curve $\gamma_{\op{SE}}^\mu$
is regular, hence 
the map
$$
\Psi:\R^3\ni (\xi,\eta,\theta)\mapsto 
\Big (\frac{\xi+\eta}2, x_1(\theta),x_2(\theta),\frac{\xi-\eta}2
\Big)
\in S^3_1
$$
is an immersion.
Since the inclusion map $\iota:\mc X_{\op{SE}}^\mu\hookrightarrow \R^3$
is a $\op{DC}$-immersion, the composition
$\tilde f_{\op{SE}}=\Psi\circ \iota$
is also a $\op{DC}$-immersion.
Then \eqref{eq:SE-E} implies that
$\tilde f_{\op{SE}}^\mu$ is an analytic extension of $f^\mu_{\op{SE}}$.
Moreover, by \eqref{eq:SE-E},
$\tilde f_{\op{SE}}^\mu$ is a real analytic
 space-like CMC-1 $\op{DC}$-immersion.

To obtain the analytic completeness of $\tilde f_{\op{SE}}^\mu$,  
it is sufficient to show that
the map $\tilde f_{\op{SE}}^\mu$ is $C^0$-arc-proper, 
since $\tilde f_{\op{SE}}^\mu$ is real analytic.
In fact, let $\Gamma:[0,1]\to S^3_1$ be a continuous
map, and suppose that
$$
\sigma(t):=(\xi(t),\eta(t),\theta(t))\qquad (t\in [0,1))
$$
satisfies $\Gamma(t)=f\circ \sigma(t)$ for each
$t\in [0,1)$.
Since the second and third components of
$\Gamma(t)$ are $\xi(t)\pm \eta(t)$,
the limits
$\dy \lim_{t\to 1-0}\xi(t)$ and
$\dy \lim_{t\to 1-0}\eta(t)$w
exist. Since
$$
\xi(t)\eta(t)+1=|\gamma_{\op{SE}}^\mu(\theta(t))|^2
=1+\left(\frac{\mu^2-1}{2\mu}\right)^2\cos^2 \mu \theta(t)
\qquad (t\in [0,1)),
$$
the limit $\dy\lim_{t\to 1-0}\cos \mu \theta(t)$ exists,
and so  
$\dy\lim_{t\to 1-0}\theta (t)$ also exists (applying Lemma B.1 in the 
second appendix).
Thus $\dy\lim_{t\to 1-0} \sigma(t)$ exists, which implies
the weak $C^0$-arc-properness of $f$.
So the conclusion follows from
Theorem \ref{thm:first}.
\end{proof}

The map $\tilde f_{\op{SE}}^\mu$ is not a proper map, since its image is not closed.
However, if $\mu\in \R_+\setminus\{1\}$ is 
an integer, 
$\tilde f_{\op{SE}}^\mu$
induces a map 
$
\hat f_{\op{SE}}^\mu:\hat {\mc X}_{\op{SE}}^{\mu}\to S^3_1,
$
where
$$
\hat {\mc X}_{\op{SE}}^\mu:=\{(\xi,\eta,\theta)\in \R^2\times T^1
\,;\,-\xi\eta+|\gamma_{\op{SE}}^\mu(\theta)|^2=1\}
$$
is a real analytic $2$-dimensional $\op{DC}$-manifold 
with finitely many $\op{DC}$-points. 
It can be easily checked that
$\hat f_{\op{SE}}^\mu$ is a proper map, and its image
is congruent to the image of the exceptional W-catenoid 
$f^{\mbox{\tiny II}}_{\mu}$ given
in \cite{FKKRUY3}.

\subsubsection{G-catenoids of type SH}
For each
 $\nu\in \R_+$, we consider
a map
\begin{equation*}
    f_{\op{SH}}^\nu :
\R^2\ni (s,\theta) \mapsto 
(x_0(s,\theta),x_1(\theta),x_2(\theta),x_3(s,\theta))
\in S^3_1, 
\end{equation*}
defined by
\begin{align*}
f^{\nu}_{\mathrm{SH}} (s, \theta)
=
\begin{pmatrix}
 \cosh s & 0 & 0 & \sinh s \\ 
 0 & 1 & 0 & 0 \\ 
 0 & 0 & 1 & 0 \\ 
 \sinh s & 0 & 0 & \cosh s
\end{pmatrix}
\Gamma^{\nu}_{\op{SH}}(\theta), \quad 
\Gamma^{\nu}_{\op{SH}}(\theta) :=
\begin{pmatrix}
 \frac{\nu ^2+1}{2 \nu} \sinh \nu \theta \\ 
x_1(\theta) \\
x_2(\theta) \\
0
\end{pmatrix}, 
\end{align*}
where $\nu\in \R_+$ and 
\begin{equation}\label{eq:x1x2-SH}
  \begin{aligned}
x_1(\theta)&=\frac{(\nu^2-1) \cos \theta  \sinh \nu\theta}{2\nu}+\sin \theta \cosh \nu\theta,\\
x_2(\theta)&=\frac{(\nu^2-1) \sin \theta  \sinh \nu\theta}{2 \nu}-\cos \theta  \cosh \nu\theta.
 \end{aligned}
\end{equation}
Here  $\Gamma^{\nu}_{\op{SH}}(\theta)$ ($\theta \in \R$) is the profile curve
of
$f^{\nu}_{\op{SH}}$
lying on $S^2_1 = S^3_1 \cap \{ x_3=0 \}$.
Then $f^{\nu}_{\mathrm{SH}}$ gives 
a G-catenoid of type {\rm SH}. 
The axis of $f_{\op{SH}}^\nu$ is the same as the axis of $f_{\op{SE}}^\mu$.

The singular point set of $f_{\op{SH}}^\nu$ is 
$\{\theta=0\}$, and its image consists of one point 
$(0,0,-1,0)$.
The limit point set of  $f_{\op{SH}}^\nu$ is 
$
	  \mc L:= \{(u,0,-1,\pm u)\,;\,u\in\R\}\subset S^3_1.
$
  Let
  \begin{equation}\label{eq:X-SH}
	  \mc X_{\op{SH}}^\nu:=\left\{(\xi,\eta,\theta)\in \R^3
\,;\,\xi\eta=\left(\frac{(\nu^2+1) 
\sinh \nu\theta}{2\nu}\right)^2\right\}\subset\R^3.
  \end{equation}
Like as for $\mc X_{\op{SE}}^\mu$,
the set $\mc X_{\op{SH}}^\nu$
has the structure of a $\op{DC}$-submanifold of $\R^3$,
whose $\op{DC}$-point is $(0,0,0)$.
Using the functions in \eqref{eq:x1x2-SH}, we set
$$
	  \tilde f_{\op{SH}}^\nu\colon{}\mc X_{\op{SH}}^\nu\ni 
(\xi,\eta,\theta)
          \mapsto \left(\frac{\xi+\eta}{2},x_1(\theta),x_2(\theta),
                        \frac{\xi-\eta}{2}\right)\in S^3_1.
$$
Then we have 
$$
f_{\op{SH}}^\nu(s,\theta)=
\tilde f_{\op{SH}}^\nu\left(\frac{(\nu^2+1)e^s}{\nu}
\sinh(\nu \theta),\frac{(\nu^2+1)e^{-s}}{\nu}
\sinh(\nu \theta),\theta\right).
$$
By setting
$$
\check f_{\op{SH}}^\nu(s,\theta):
=\tilde f_{\op{SH}}^\nu\left(
-\frac{(\nu^2+1)e^s}{\nu}\sinh(\nu \theta),-\frac{(\nu^2+1)e^{-s}}{\nu}
\sinh(\nu \theta),\theta\right),
$$
we have
$$
\check f_{\op{SH}}^\nu(s,\theta)=
\pmt{
-1 & 0 & 0 & 0 \\
0 & 1 & 0 & 0 \\
0 & 0 & 1 & 0 \\
0 & 0 & 0 & -1 
}f^\nu_{\op{SH}} (s, \theta),
$$
which implies the image of $\check f_{\op{SH}}^\nu$ is 
congruent to that of $f_{\op{SH}}^\nu$. 
It holds that 
\begin{equation}\label{eq:SH-I}
\tilde f_{\op{SH}}^\nu(\mc X_{\op{SH}}^\nu)=
f^\nu_{\op{SH}}(\R^2)\cup \mc  L\cup \check f_{\op{SH}}^\nu(\R^2).
\end{equation}
We prove the following: 

\begin{Lemma}\label{lem:5-4}
The plane curve $\theta\mapsto (x_1(\theta),x_2(\theta))$ is a 
regular curve on $\R$.
\end{Lemma}

\begin{proof}
In fact, the assertion follows from the following identity
$$
\pmt{x_1'(\theta) \\ x'_2(\theta)}=
\frac{\nu^2+1}{2\nu}
\pmt{ \cos \theta & -\sin \theta \\
  \sin \theta  & \cos \theta
}
\pmt{
\nu \cosh \nu \theta \\  -\sinh \nu \theta
}.
$$
\end{proof}
By this lemma, we can conclude that
$$
\Psi:\R^3\ni (\xi,\eta,\theta)
          \mapsto \left(\frac{\xi+\eta}{2},x_1(\theta),x_2(\theta), 
          \frac{\xi-\eta}{2}\right)
\in S^3_1
$$
is an immersion.
Since the inclusion map $\iota:\mc X_{\op{SH}}^\nu\hookrightarrow \R^3$
is a $\op{DC}$-immersion, the composition
$\tilde f_{\op{SH}}=\Psi\circ \iota$
is also a $\op{DC}$-immersion. 
The relation \eqref{eq:SH-I} implies that
$\tilde f_{\op{SH}}^\nu$ is an analytic extension of $f_{\op{SH}}^\nu$.
Moreover, by Proposition \ref{prop:AB-1},
$\tilde f_{\op{SH}}^\nu$
is a 
real analytic space-like CMC-1 $\op{DC}$-immersion.
We set $\gamma_{\op{SH}}^\nu(\theta)=(x_1(\theta), x_2(\theta))$. 
Since
 \[
   |\gamma_{\op{SH}}^\nu(\theta)|^2= 1+\left(\frac{\nu^2+1}{2\nu}
\right)^2\sinh^2 \nu \theta,
 \]
the function $|\gamma_{\op{SH}}^\nu(\theta)|$ diverges as $\theta\to \pm \infty$.
Since the curve $\gamma_{\op{SH}}^\nu$ is a 
proper map, $\tilde f_{\op{SH}}^\nu$ is also a proper map.
So we can apply Theorem \ref{thm:first} and obtain
the following:

\begin{Proposition}\label{prop:SH-noext}
The map
$\tilde f_{\op{SH}}^\nu$
is a 
real analytic space-like CMC-1 $\op{DC}$-immersion
whose
image is analytically complete.
\end{Proposition}

\subsubsection{G-catenoids of type SP}
 A G-catenoid of type {\rm SP} 
can be expressed as
\begin{equation*}
   f_{\op{SP}} \colon \R^2 \ni (s,t) \mapsto 
(x_0(s,t),x_1(t),x_2(t),x_3(s,t))  \in S^3_1,
\end{equation*}
where
\begin{align*}
f_{\mathrm{SP}}(s,t)
=
\begin{pmatrix}
 \cosh s & 0 & 0 & \sinh s \\ 
 0 & 1 & 0 & 0 \\ 
 0 & 0 & 1 & 0 \\ 
 \sinh s & 0 & 0 & \cosh s
\end{pmatrix}
\Gamma_{\op{SP}}(t), \quad 
\Gamma_{\op{SP}}(t) := 
\begin{pmatrix}
 -\frac{t}{2} \\ \frac{1}{2} t \cos t -\sin t \\ 
\frac{1}{2} t \sin t+\cos t \\ 0
\end{pmatrix}. 
\end{align*}
Here  $\Gamma_{\op{SP}}(t)$ ($t \in \R$) is the profile curve
of 
$f_{\op{SP}}$
lying on $S^2_1 = S^3_1 \cap \{ x_3=0 \}$.
The axis of $f_{\op{SP}}$ is the same as the axis of $f_{\op{SE}}^\mu$.
The singular point set of $f_{\op{SP}}$ is 
       $\{t=0\}$, and its image consists of one point 
       $(0,0,1,0)$. On the other hand,
the limit point set of  $f_{\op{SP}}$ in $S^3_1$ is 
$\mc L:= \{(u,0,1,\pm u)\,;\,u\in\R\}$.
This map has an analytic extension as follows:
Like as
$\mc X_{\op{SE}}^\mu$ and
$\mc X_{\op{SH}}^\nu$,
the set 
  \begin{equation}\label{eq:X-SP}
	  \mc X_{\op{SP}}:=
\{(\xi,\eta,t)\,;\,\xi\eta=t^2/4\}\subset\R^3
  \end{equation}
is a $\op{DC}$-submanifold in $\R^3$
admitting only one cone point at $(0,0,0)$.
We can define a real analytic $\op{DC}$-map on the $\op{DC}$-manifold 
$\mc X_{\op{SP}}$ by
$$
	  \tilde f_{\op{SP}}\colon{}\mc X_{\op{SP}}\ni
 (\xi,\eta,t)
          \mapsto \frac12\Big(\xi+\eta,t \cos t-2 \sin t,
t \sin t+2 \cos t,
                        \xi-\eta\Big)\in S^3_1.
$$
Then, we have
$
f_{\op{SP}}(s,t)=\tilde f_{\op{SP}}
\left(-{e^st}/2,-{e^{-s}t}/2,t\right).
$
If we set
$$
\check f_{\op{SP}}(s,t):=
\tilde f_{\op{SP}}
\left(\frac{e^st}2,\frac{e^{-s}t}2,t\right),
$$
then
\begin{equation}\label{eq:SP-S}
\check f_{\op{SP}}=
\pmt{
-1 & 0 & 0 & 0 \\
0 & 1 & 0 & 0 \\
0 & 0 & 1 & 0 \\
0 & 0 & 0 & -1 
}
f_{\op{SP}}
\end{equation}
and
\begin{equation}\label{eq:SP-I}
\tilde f_{\op{SP}}(\mc X_{\op{SP}})=f_{\op{SP}}(\R^2)\cup \mc L\cup
\check f_{\op{SP}}(\R^2).
\end{equation}
The relation \eqref{eq:SP-S}
implies that $\check f_{\op{SP}}$ is congruent to $f_{\op{SP}}$. 
It can be easily checked that $\R\ni t \mapsto  (x_1(t), x_2(t))\in \R^2$
is an immersion, and
$$
\Psi:\R^3 \ni 
 (\xi,\eta,t)
          \mapsto \frac12\Big(\xi+\eta,t \cos t-2 \sin t,
t \sin t+2 \cos t,
                        \xi-\eta\Big)\in S^3_1 
$$
also gives an immersion. 
Since the inclusion map $\iota:\mc X_{\op{SP}}\hookrightarrow \R^3$
is a $\op{DC}$-immersion, the composition
$\tilde f_{\op{SP}}=\Psi\circ \iota$
is also a $\op{DC}$-immersion.
On the other hand,
since 
$\R\ni t \mapsto  (x_1(t), x_2(t))\in \R^2$ 
is a proper map,
so is $\tilde f_{\op{SP}}$.
In particular, we obtain the following:

\begin{Proposition}\label{prop:SP-noext}
The map $\tilde f_{\op{SP}}$ 
is a
real analytic space-like CMC-1 $\op{DC}$-immersion
whose
image is analytically complete.
\end{Proposition}

\subsection{G-catenoids of type L}

\subsubsection{G-catenoids of type LH}
A G-catenoid of type {\rm LH}  
can be expressed as
\begin{equation*}
   f_{\op{LH}} \colon{}\R^2
\ni (u,v) \mapsto P(u)\Gamma_{\op{LH}}(v) \in S^3_1,  
\end{equation*}
where 
\begin{equation*}
 P(u)= 
 \begin{pmatrix}
 1+\frac{u^2}{2} & u & 0 & -\frac{u^2}{2} \\ 
 u & 1 & 0 & -u \\ 
 0 & 0 & 1 & 0 \\ 
 \frac{u^2}{2} & u & 0 & 1 -\frac{u^2}{2}
 \end{pmatrix}, 
\quad
\Gamma_{\op{LH}}(v) := 
\begin{pmatrix}
 v \cos 2v+\frac{1}{2} v^2 \sin 2v \\ 
0 \\ 
\cos 2v + v \sin 2v \\
 v \cos 2v+\frac{1}{2} (v^2-2) \sin 2v
\end{pmatrix}, 
\end{equation*}
that is, $f_{\op{LH}}=(x_0,x_1,x_2,x_3)$ is given by 
\begin{equation}\label{eq:f-LE}
 \begin{aligned}
x_0&= v \cos 2 v + \frac{u^2+v^2}{2} \sin 2v, &
x_1&= u \sin 2  v, \\ 
x_2&= y_{\op{LH}}(v):=\cos 2v + v \sin 2v, &
x_3&= v \cos 2 v + \frac{u^2+v^2-2}{2} \sin 2v. 
\end{aligned}
\end{equation}
This $f_{\op{LH}}$ is generated by 
the profile curve $\Gamma_{\op{LH}}(v)$ ($v \in \R$)
lying on $S^2_1 = S^3_1 \cap \{ x_1=0 \}$. 
The singular point set is $\{2v\equiv 0\pmod \pi\}$,
       and the image of each connected component of the singular
       set consists of one point.
The limit point set of  $f_{\op{LH}}\colon{}\R^2\to S^3_1$ 
       consists of two light-like lines
$	 \mc L_{\pm}:=\{(t,0,\pm 1,t)\in S^3_1\,;\, t\in\R\}
$, which are the axes of $f_{\op{LH}}$.
The figure of the image of $f_{\op{LH}}$ is complicated 
(cf. Figure \ref{fig:typeL}, left).

\begin{figure}[hbt]%
 \begin{center}
       \includegraphics[width=3.5cm]{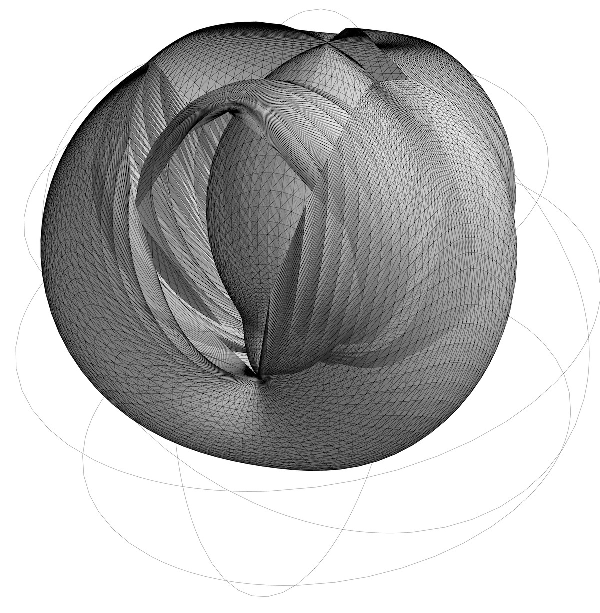} \hskip 2cm
       \includegraphics[width=3.5cm]{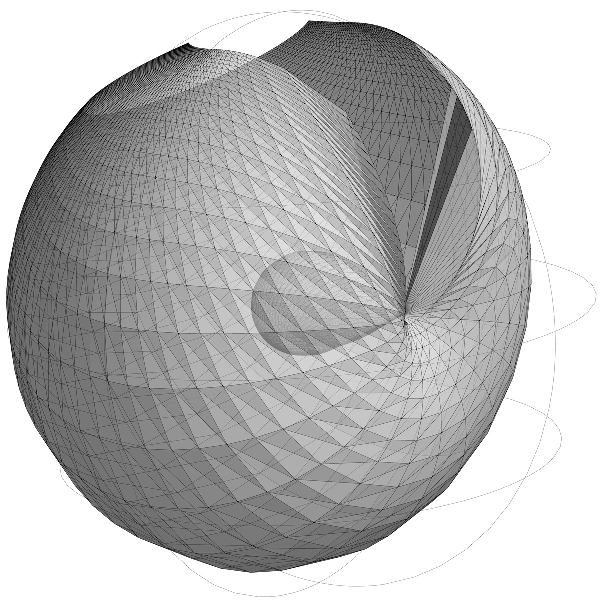} 
 
\end{center}
  \caption{G-catenoids of type LH (left) and LE (right)}%
\label{fig:typeL}
\end{figure}

  Let
  \begin{equation}\label{eq:X-LH}
	  \mc X_{\op{LH}}:=\left\{(x,\xi,v)\,;\,-\xi\sin 2v +x^2 
+(\cos^2 2v -v \sin^2 2v)^2-1=0\right\}\subset\R^3,
  \end{equation}
and set
$$
	  \tilde f_{\op{LH}}\colon{}\mc X_{\op{LH}}\ni
 (x,\xi,v)
          \mapsto \left(\frac{\xi+\sin 2v}{2},x,y^{}_{\op{LH}}(v),
                        \frac{\xi-\sin 2v}{2}\right)\in S^3_1.
$$
Then it holds that
$$
f_{\op{LH}}(u,v)=\tilde f_{\op{LH}}
\left(\frac{(u^2+v^2+1)\sin 2v}2,u \sin 2v,v\right)
$$
and
\begin{equation}\label{eq:LH-I}
\tilde f_{\op{LH}}(\mc X_{\op{LH}})
= f_{\op{LH}}(\R^2)\cup \mc L_+\cup \mc L_-.
\end{equation}

\begin{Lemma}\label{lem:LH}
The subset $\mc X_{\op{LH}}$ 
has the structure of
a $\op{DC}$-submanifold of $\R^3$
such that 
the residual set is the set of $\op{DC}$-points
and the inclusion map $\mc X_{\op{LH}} \hookrightarrow \R^3$
is a $\op{DC}$-immersion.
\end{Lemma}

\begin{proof}
Here $(x,\xi,v)=(0,0,\pi k)\in \R^3$ ($k\in \Z$)
are the 
singular points of $\mc X_{\op{LH}}$.
We set $\phi:=\sqrt{(1+v)(2-(1+v)\sin^2(2v))}$,
which can be locally considered as a function of
$\sin 2v$, so we denote it by 
$\phi(w)$ ($w:=\sin 2v$).
Then, we can write
$$
-\xi\sin 2v +(\cos^2 2v -v \sin^2 2v)^2-1
=\left(\frac{\xi}{2\phi(w)}\right)^2-
\left(\frac{\xi}{2\phi(w)}+w \phi(w)\right)^2.
$$
So if we set
$$
Y:=\frac{\xi}{2\phi(w)},\qquad Z:=\frac{\xi}{2\phi(w)}+w \phi(w),
$$
then we have
$$
-\xi\sin 2v +\left(\cos^2 2v -v \sin^2 2v\right)^2-1=
Y^2-Z^2.
$$
Here,
$$
\Phi(\xi,w,x):=(x,Y(\xi,w),Z(\xi,w))
$$
gives a local $C^\omega$-diffeomorphism around $(0,0,0)$.
By identifying $v+k\pi$ with $w$
by a local $C^\omega$-diffeomorphism $v \mapsto \sin v$,
the local inverse map
$\Phi^{-1}$ can be considered as a 
$C^\omega$-diffeomorphism
from a neighborhood of the origin of $\R^3$ to
a neighborhood of $(x,\xi,v)=(0,0,\pi k)$,
which maps the cone $x^2+Y^2-Z^2=0$ near $(0,0,0)$
to an open subset of $\mc X_{\op{LH}}$. 
Thus, by Proposition \ref{prop:Xstr},
we obtain the conclusion.
\end{proof}

By
\eqref{eq:LH-I},
$\tilde f_{\op{LH}}$ is an analytic extension of
$f_{\op{LH}}$, and is a 
real analytic space-like CMC-1 $\op{DC}$-immersion,
giving an analytic extension of the image of $f_{\op{LH}}$.
It can be easily checked that
$$
\R^3\ni 
 (x,\xi,v)
          \mapsto \left(\frac{\xi+\sin 2v}{2},x,y^{}_{\op{LH}}(v),
                        \frac{\xi-\sin 2v}{2}\right)\in S^3_1
$$
is an immersion. Since  $\tilde f_{\op{LH}}$ is 
the restriction of this map to $\mc X_{\op{LH}}$,
it is a $\op{DC}$-immersion.
The map  $\tilde f_{\op{LH}}$ is not proper. 
In fact, the sequence
$\{(0,0,k\pi)\}_{k=1}^\infty$ lies in $\mc X_{\op{LH}}$
and diverges in $\R^3$, although its image under $\tilde f_{\op{LH}}$ is
the point set $\{(0,0,1,0)\}$.
We prepare the following:

\begin{Lemma}
The set  $\mc X_{\op{LH}}(\subset \R^3)$
has the structure of a real analytic $2$-dimensional $\op{DC}$-manifold
with a countably infinite number of  $\op{DC}$-points.
Moreover, 
$\tilde f_{\op{LH}}$ is a $\op{DC}$-immersion,
giving the analytic extension of $f_{\op{LH}}$. 
\end{Lemma}

\begin{proof}
We set $F:=-\xi\sin 2v +x^2 +y_{\op{LH}}(v)^2-1$.
It can be easily checked that
$F=F_{\xi}=F_x=F_v=0$ if and only if
$$
v=v_0,\quad \xi=2\epsilon_0 v_0, \quad  x=0 
\qquad \left(\epsilon_0:=\cos(2v_0)\in \{\pm 1\},\,\,v_0\in \frac{\pi}2\Z\right).
$$
So we fix such a  $v_0$.
We set $\tilde v:=\sin 2v$, and
then we can write
$$
v-v_0=\tilde v\phi(\tilde v), \qquad 
\cos 2v=\epsilon_0 +\tilde v\psi(\tilde v),
$$
where $\phi(\tilde v)$ and $\psi (\tilde v)$ are  
functions defined on a neighborhood of $\tilde v=0$.
So if we set
$$
\tilde \xi:=\xi-\xi_0+\tilde v (\phi(\tilde v)^2
+2\epsilon_0 \phi(\tilde v) + 1),
$$
then we can write  $F=-\tilde \xi \tilde v +x^2$.
Using this expression, one can conclude that 
the zero set of $F$ has a 
cone-like singular point
at $(\xi_0,0,v_0)$
of $\tilde f_{\op{LH}}$.
Since the map
$$
S^3_1\ni (x_0,x_1,x_2,x_3)  \mapsto \Big(x_0+x_3, x_1,
\frac{\sin^{-1}(x_0-x_3)}2\Big ) \in \R^3
$$
gives a local inverse map of the immersion
$$
\R^3\ni  (\xi,x,v)
          \mapsto \left(\frac{\xi+\sin 2v}{2},x,y^{}_{\op{LH}}(v),
                        \frac{\xi-\sin 2v}{2}\right)\in S^3_1,
$$
$\tilde f_{\op{LH}}$ is a $\op{DC}$-immersion.
By \eqref{eq:LH-I}, it also gives an analytic extension of
$f_{\op{LH}}$. 
\end{proof}

\begin{Proposition}\label{prop:LH-noext}
The map  
$\tilde f_{\op{LH}}$ is a real analytic
 space-like CMC-1 $\op{DC}$-immersion
 whose
image is analytically complete.
\end{Proposition}

\begin{proof}
Applying Proposition \ref{prop:AB-1},
we can conclude that
$\tilde f_{\op{LH}}$ is a 
real analytic space-like CMC-1 $\op{DC}$-immersion.
The map $\tilde f_{\op{LH}}$ is not a proper map,
since $\tilde f_{\op{LH}}(\xi,0,2m\pi)=(\xi/2,0,1,\xi/2)$
for $m\in \Z$. 
By 
Theorem \ref{thm:first}, 
it is sufficient to show that 
$\tilde f_{\op{LH}}$ is $C^0$-arc-proper.
We consider a continuous curve
$
\Gamma:[0,1]\to S^3_1
$
such that $\Gamma([0,1))\subset \tilde f_{\op{LH}}(\mc X_{\op{LH}})$.
 Let  $\sigma(t)=(\xi(t),x(t),v(t))$ ($0\le t<1$)
be a continuous
curve on $\mc X_{\op{LH}}$
such that 
$$
\tilde f_{\op{LH}}\circ \sigma(t)=\Gamma(t) \qquad
(0\le t<1).
$$
Since $\Gamma(1)$ exists, the three limits
$$
\lim_{t\to 1-0}\xi(t), \quad \lim_{t\to 1-0} x(t),\quad
\lim_{t\to 1-0} \sin 2v(t)
$$
exist. 
Then
$\dy\lim_{t\to 1-0} v(t)$ also exists (cf. Lemma B.1 in the 
second appendix),
proving the assertion.
\end{proof}

\subsubsection{G-catenoids of type LE}
On the other hand,
a G-catenoid of type {\rm LE}  
can be expressed as
\begin{equation*}
   f_{\op{LE}} \colon{}\R^2
\ni (u,v) \mapsto P(u)\Gamma_{\op{LE}}(v) \in S^3_1,  
\end{equation*}
where 
\begin{equation*}
 P(u)= 
 \begin{pmatrix}
 1+\frac{u^2}{2} & u & 0 & -\frac{u^2}{2} \\ 
 u & 1 & 0 & -u \\ 
 0 & 0 & 1 & 0 \\ 
 \frac{u^2}{2} & u & 0 & 1 -\frac{u^2}{2}
 \end{pmatrix}, 
\quad
\Gamma_{\op{LE}}(v) := 
\begin{pmatrix}
 v \cosh 2v-\frac{1}{2} (v^2+2) \sinh 2v \\ 
0 \\ 
\cosh 2v - v \sinh 2v \\
 v \cosh 2v-\frac{1}{2} v^2 \sinh 2v
\end{pmatrix}, 
\end{equation*}
that is, $f_{\op{LE}}=(x_0,x_1,x_2,x_3)$ is given by 
\begin{equation}\label{eq:f-LH}
 \begin{aligned}
x_0&= v \cosh 2v - \frac{u^2+v^2+2}{2} \sinh 2v, &
x_1&= -u \sinh 2v, \\ 
x_2&= y_{\op{LE}}(v):=\cosh 2v - v \sinh 2v, &
x_3&= v \cosh 2 v - \frac{u^2+v^2}{2} \sinh 2v. 
\end{aligned}
\end{equation}
Here  $\Gamma_{\op{LE}}(v)$ ($v \in \R$) is a profile curve
of
$f_{\op{LE}}$
lying on $S^2_1 := S^3_1 \cap \{ x_1=0 \}$.
The axes of $f_{\op{LE}}$ are the same as the axes of $f_{\op{LH}}$.
The singular point set is  $\{v=0\}$, whose 
image consists of one point $(0,0,1,0)$.
The limit point set of $f_{\op{LE}}\colon{}\R^2\to S^3_1$ consists
of a light-like line
       $
	 \mc L:=\{(t,0,1,t)\in S^3_1\,;\, t\in\R\}.
       $
Set
\begin{equation}\label{eq:X-LE}
\mc X_{\op{LE}}:=\Big\{(\xi,x,v)\,;\,\xi
\sinh 2 v+x^2 +y_{\op{LE}}^{}(v)^2=1\Big\}\subset\R^3,
 \end{equation}
and let
$$
	  \tilde f_{\op{LE}}\colon{}\mc X_{\op{LE}}\ni 
(\xi,x,v)
          \mapsto \left(\frac{\xi-\sinh 2v}{2},x,y_{\op{LE}}^{}(v),
                        \frac{\xi+\sinh 2v}{2}\right)\in S^3_1.
$$
Since the expression of $\mc X_{\op{LE}}$ is 
obtained by replacing $\sin 2v$ and $\cos 2v$ by
$\sinh 2v$ and $\cosh 2v$,
like as in Lemma \ref{lem:LH},
one can show that
the set  $\mc X_{\op{LE}}(\subset \R^3)$
has the structure of a real analytic $2$-dimensional $\op{DC}$-manifold
with a $\op{DC}$-point, and
the residual set is the set of $\op{DC}$-points
and the inclusion map $\mc X_{\op{LE}} \hookrightarrow \R^3$
is a $\op{DC}$-immersion.
Moreover, 
it can be easily checked that $\tilde f_{\op{LE}}$
is a $\op{DC}$-immersion defined on the
$\op{DC}$-manifold $\mc X_{\op{LE}}$,
and is a proper map.
Moreover, 
we have
$$
f_{\op{LE}}(u,v)=\tilde f_{\op{LE}}\left(v \cosh 2v - \frac{u^2+v^2+1}{2} 
\sinh 2v,-u\sinh v,v\right)
$$
and  $\tilde f_{\op{LE}}(\mc X_{\op{LE}})= f_{\op{LE}}(\R^2)\cup \mc L$.
In particular, $\tilde f_{\op{LE}}$ 
is an analytic extension of $f_{\op{LE}}$.
Since
$$
\R^3\ni(\xi,x,v)
          \mapsto \left(\frac{\xi-\sinh 2v}{2},x,y_{\op{LE}}^{}(v),
                        \frac{\xi+\sinh 2v}{2}\right)\in S^3_1
$$
is an immersion,
$\tilde f_{\op{LE}}$ is a real analytic $\op{DC}$-immersion. 
Moreover, by Proposition \ref{prop:AB-1}, it is 
a 
real analytic space-like CMC-1 $\op{DC}$-immersion
giving 
an analytic extension of $f_{\op{LE}}$. 
The image of $\tilde f_{\op{LE}}$ can be expressed as 
the intersection of $S^3_1$ and the graph of
\[
  x_2 = \sqrt{1+(x_0-x_3)^2} -\frac12 (x_0-x_3)\sinh^{-1}(x_0-x_3).
\]
By Proposition \ref{prop:first},
we obtain the following:

\begin{Proposition}\label{prop:LH-noext2}
The map $\tilde f_{\op{LE}}$ 
is a 
real analytic space-like CMC-1 $\op{DC}$-immersion
whose
image is analytically complete.
\end{Proposition}

Finally, summarizing the results in this section,
we obtained the following:

\begin{Theorem}
The images of G-catenoids of type $T$ are analytically complete
and those of
other types admit analytic extensions, which are 
analytically complete.
Moreover, after taking the analytic completions, 
all of them  are the images of
real analytic space-like CMC-1 $\op{DC}$-immersions into $S^3_1$.
\end{Theorem}

\appendix
\section{Maximal catenoids in $\R^3_1$}

This appendix gives an overview of G-catenoids
and W-catenoids which are space-like maximal surfaces in $\R^3_1$.

We give the following definition:

\begin{Definition}\label{def:ZMC}
Let $f:X^2\to \R^3_1$ be a real analytic
$\op{DC}$-immersion 
defined on a connected $2$-dimensional real analytic
$\op{DC}$-manifold $X^2$.
Then $f$ is called 
a {\it ZMC $\op{DC}$-immersion}
if 
there exists an open dense subset $O$ of
$X^2\setminus \Sigma$
such that the restriction $f|_O$
is a zero mean curvature immersion
on $O$,
where $\Sigma$ is the set of $\op{DC}$-points in $X^2$.
\end{Definition}

We first recall the property of real analytic ZMC $\op{DC}$-immersions 
from \cite{AUY} as follows:
Let $U$ be a domain in the $uv$-plane $\R^2$, and let
$f\colon U\to \R^3_1$ be a real analytic map
into the Lorentz-Minkowski 3-space $\R^3_1$
of signature $(-++)$.
We set
$$
B_f:=\det(P),\qquad
P:=
\pmt{f_{u}\cdot f_{u}\, &\, f_{u}\cdot f_{v} \\
f_{v}\cdot f_{u}\, &\, f_{v}\cdot f_{v} 
},
$$
where $\cdot$ denotes the canonical Lorentzian inner
product of
$\R^3_1$.
We set
$$
Q:=
\pmt{f_{uu}\cdot \tilde \nu\, &\, f_{uv}\cdot \tilde \nu \\
f_{vu}\cdot \tilde \nu\, &\, f_{vv}\cdot \tilde \nu 
},
$$
where $\tilde \nu:=f_u\times_L f_v$
and $\times_L$ is the canonical Lorentzian vector product
of $\R^3_1$. 
In this situation, $f$ is called a zero mean curvature map
if 
$
A_f:=\op{trace}(\tilde PQ)
$
vanishes identically, 
where $\tilde P$ is the cofactor matrix of $P$.
If $f$ is a {\it ZMC $\op{DC}$-immersion}, then
it is a zero mean curvature map.
So, imitating the proof of Proposition \ref{prop:AB-1},
we can prove the following
(although real analytic space-like CMC-1 $\op{DC}$-immersions  never change
type from space-like to  time-like, real analytic zero mean 
curvature immersions may change type):

\begin{Proposition}\label{prop:AB-0}
Let $f:X^2\to \R^3_1$ be a 
$\op{DC}$-immersion
defined on a connected $2$-dimensional real analytic
$\op{DC}$-manifold $X^2$.
Suppose that there exists a non-empty local inverse
$\op{DC}$-coordinate system $(U,\phi)$ such that 
$f\circ \phi:U\to \R^3_1$ gives a zero mean curvature immersion
on an open dense subset of $U$.
Then $f$ is a ZMC  $\op{DC}$-immersion.
\end{Proposition}

Osamu Kobayashi \cite{K} showed that
classified G-catenoids
are
\begin{itemize}
 \item the elliptic G-catenoid (cf. \eqref{eq:e-catenoid}),
  \item the parabolic G-catenoid (cf. \eqref{eq:p-catenoid}) and
\item the hyperbolic G-catenoid (cf. \eqref{eq:h-catenoid}).
\end{itemize}
The image of the elliptic G-catenoid $f_E$ 
(cf. \eqref{eq:e-catenoid})
can be considered as
the image of a ZMC  $\op{DC}$-immersion
(cf. Example \ref{ex:E0}).
The image $\mc E$ of $f_E$ is
analytically complete (cf. Theorem \ref{Cor:EPHK}).
We next consider  
the parabolic G-catenoid $f_P$ in \eqref{eq:p-catenoid}.
We let  $\mc P$ be the subset of $\R^3_1$
defined by \eqref{eq:p2}.
Then $\mc P$ gives an analytic extension of $f_P$,
which is analytically complete
(cf. Theorem \ref{Cor:EPHK}).
On the other hand,
the analytic extension 
$\mc H$ 
of the image of $f_H$
is given by \eqref{eq:H0},
which is a singly periodic maximal surface with
cone-like singular points at
$\Sigma:=\{(0, n\pi, 0)\,;\, n\in \Z\}$
with respect to the inclusion map:
As shown in Theorem \ref{Cor:EPHK},
$\mc H$ is analytically complete.

It is known that
there are only two
 congruence classes of W-catenoids up to a homothety
(cf.~Imaizumi-Kato \cite{IK}).
One is represented by the elliptic G-catenoid $f_E$, which is 
analytically complete, as mentioned above.
The other is represented by the maximal surface
$f_K$
given in
\eqref{eq:A-cat-max},
which is the Kobayashi surface of order $2$ of type 
$(0,0, \pi,\pi)$ in \cite{FKKRUY4}.
This surface has an analytic extension 
as the graph of
$t=x \tanh y$,
which is analytically complete
(cf. Theorem \ref{Cor:EPHK}).

\section{A property of continuous functions}
\label{app:continuous}

In this appendix, we point out the following 
property of continuous functions, which is used
in Section 4:

\begin{Lemma}\label{lem:C}
Let $f:\R\to\R$ be a non-constant 
real analytic periodic function,
and let $\phi:[0,1)\to \R$ be a continuous function.
If 
$
\dy\lim_{t\to 1-0}f\circ \phi(t)
$ 
exists, then $\dy\lim_{t\to 1-0}\phi(t)$ also exists.
\end{Lemma}

\begin{proof}
Since $f$ is non-constant,
there exists a closed interval $[a,b]$ ($a<b$)
which coincides with $f(\R)$.
If $\phi(t)$ is unbounded as $t\to 1-0$, 
then $f\circ \phi(t)$ takes all values in $[a,b]$
infinitely many times, contradicting the existence of the
limit $\dy\lim_{t\to 1-0}f\circ \phi(t)$. 
So $\phi(t)$ is bounded.
Suppose that $\dy\lim_{t\to 1-0}\phi(t)$ does not exist and so 
there exist two distinct accumulation points 
$\alpha,\beta\in \R$ ($\alpha<\beta$)
of $\phi$. 
Then there are two distinct sequences 
$\{t_k\}_{k=1}^\infty$ and
$\{s_k\}_{k=1}^\infty$ in $[0,1)$ 
converging to $1$ such that
$$
\lim_{k\to \infty}\phi(t_k)=\alpha,\quad
\lim_{k\to \infty}\phi(s_k)=\beta.
$$
By taking a subsequence of $\{s_k\}_{k=1}^\infty$ 
and a subsequence of $\{t_k\}_{k=1}^\infty$ if necessary,
we may assume that $t_k<s_k<t_{k+1}$
for each positive integer $k$.
We can find a positive number $\delta(<(\beta-\alpha)/3)$
and a positive integer $N$ such that
$$
|\phi(t_k)-\alpha|,\,\,|\phi(s_k)-\beta|<\delta
\qquad (k\ge N).
$$
Then 
$
[\alpha+\delta,\beta-\delta]\subset \phi([t_k,s_k])
$
holds for each $k(\ge N)$.
We denote by $M$ (resp. $m$) the maximum (resp. minimum)
value of $f$ on the interval $[\alpha+\delta,\beta-\delta]$.
Since $f$ is a non-constant
real analytic function, we have  $m<M$. 
Then, by the intermediate value theorem,
there exist $u_k,v_k\in (t_k,s_k)$
such that $f\circ \phi(u_k) =m$ 
and $f \circ \phi (v_k) = M$ for each $k$.
Since $s_k$ and $t_k$ are converging to $1$,
so are the two values $u_k$ and $v_k$. 
Then 
$$
m=\lim_{k\to \infty}f\circ \phi(u_k)=
\lim_{k\to \infty}f\circ \phi(v_k)=M,
$$
a contradiction. Hence, $\dy\lim_{t\to 1-0}\phi(t)$ exists. 
\end{proof}

\begin{acknowledgements}
The authors thank the referee, Professors 
Toshizumi Fukui and Udo Hertrich-Jeromin for valuable comments and
suggestions.
\end{acknowledgements}

\end{document}